\documentclass{amsart}
\usepackage{amsfonts,amscd, amssymb}

\newtheorem{theorem}{Theorem}[section]
\newtheorem{lemma}[theorem]{Lemma}
\newtheorem{corollary}[theorem]{Corollary}
\newtheorem{proposition}[theorem]{Proposition}
\theoremstyle{remark}

\theoremstyle{definition}

\numberwithin{equation}{section}
\makeatother

\DeclareMathOperator{\Cdb}{{\mathbb C}}
\DeclareMathOperator{\Rdb}{{\mathbb R}}

\DeclareMathOperator{\Ndb}{{\mathbb N}}

\begin{document}

\title[Contractive projections on operator algebras]{Contractive projections and real positive maps on  operator algebras}
\author{David P. Blecher}
\address{Department of Mathematics, University of Houston, Houston, TX
77204-3008, USA}
\email[David P. Blecher]{dblecher@math.uh.edu}

 \author{Matthew Neal}
\address{Math and Computer Sci.\ Department, Denison University, Granville, OH 43023,
 USA}
\email[Matthew Neal]{nealm@denison.edu}

\date{11/7/2019} 
\thanks{DB  is supported by a Simons Foundation Collaboration Grant.  We are also grateful to Denison University for some financial support}

\begin{abstract}  We study contractive  projections, isometries, and real positive maps  on algebras of operators on a Hilbert space.  For example 
we find generalizations and variants of certain classical results on contractive projections
on $C^*$-algebras and JB-algebras due to Choi, Effros, St{\o}rmer, Friedman and Russo, and others.   In fact most of our arguments generalize to contractive `real positive' projections on
 Jordan operator algebras, that is on  a norm-closed space  $A$ of operators
on a Hilbert space with $a^2 \in A$ for all $a \in A$.   We also  prove many new general 
results on real positive maps which are foundational to the study of such maps, and of interest in their own right. We also prove a new 
Banach-Stone type theorem for isometries between operator algebras 
or Jordan operator algebras.  An application of this
is given to the characterization of symmetric real positive projections.
  \end{abstract}

\maketitle

\section{Introduction}

An (associative) {\em operator algebra} is a possibly nonselfadjoint closed subalgebra $A$ of $B(H)$, for a
complex Hilbert space $H$.     Here we study contractive  projections, isometries, and real positive maps  on
such algebras, and on their `nonassociative counterparts'.  
By a projection on a Banach space $X$ we mean an idempotent (usually contractive) linear map $P : X \to X$.   In a previous paper \cite{BNp} we studied completely contractive projections and conditional expectations 
on such  operator algebras, and in particular we found variants of certain deep classical results on contractive 
positive projections
on $C^*$-algebras and JB-algebras due to Choi, Effros, St{\o}rmer, 
Friedman and Russo, and others.   In the present paper we attempt to generalize our  work
from \cite{BNp} 
to  the setting of contractive projections.   Studying  contractive projections or isometries of operator  algebras forces 
one into the more general  setting of Jordan operator algebras, that is to a  norm-closed space of operators $A$
on a Hilbert space closed under the 
`Jordan product' $a \circ b = \frac{1}{2}(ab+ba)$ (or equivalently, with  $a^2 \in A$ for all $a \in A$). 
 For example, the range of a positive unital projection on a $C^*$-algebra need not be again be isometrically
isomorphic to a $C^*$-algebra
(consider $\frac{1}{2}(x + x^\intercal )$ on $M_2$), but it is always  isometrically
isomorphic to a Jordan operator algebra.     Also,  as one sees already in 
Kadison's Banach-Stone  theorem for
$C^*$-algebras \cite{Kad}, isometries of $C^*$-algebras relate to Jordan $*$-homomorphisms and not necessarily to 
$*$-homomorphisms.    Thus  most of our results  are  best stated
for, or belong naturally to,  the larger category of Jordan operator algebras.  However a few  of our results 
will apply only to (associative) operator  algebras.  

  To establish our results, as in \cite{BNp} we  add the ingredient of `real positivity' from recent 
papers of the first author with Read 
in \cite{BRI,BRII,BRord} (see also e.g.\ \cite{BNII,BBS,BNp,Bsan,BWj,BNj}).     
A key idea in those papers is that  `real positivity' is often the right replacement in general algebras for positivity
in $C^*$-algebras.   
Thus we will be using for our `positive cone' the real positive operators, 
or operators with positive real part; 
namely the operators $T$  satisfying $T + T^* \geq 0$.   Sometimes these operators are called  {\em accretive}.
(We remark that there was already some use of accretive operators in the `JB-algebra literature'; see particularly the use of 
`dissipativity' in e.g.\ \cite{Rod, Rod2}.)  
 This will be our  guiding principle here. 
In an early section of our paper we prove many new and foundational 
results of independent interest concerning  real positive maps, that is, maps which take  real positive elements to  real positive elements.
This part of our paper is a generalization of aspects of the basic theory of positive maps on $C^*$-algebras \cite{ST}.
We also prove a new 
Banach-Stone type theorem for isometries between 
 operator algebras 
or Jordan operator algebras: a characterization
of such isometries in the spirit of Kadison's Banach-Stone  theorem for
$C^*$-algebras.   One of the difficulties to be overcome here is the fact that there exist linearly isometric unital $JC^*$-algebras which are not Jordan isomorphic  (see e.g.\ \cite[Section 5]{BKU} and  \cite[Antitheorem 3.4.34]{Rod}). Therefore  
Banach-Stone theorems for nonunital isometries between Jordan operator algebras cannot exactly have the form one would first think of
(see the second paragraph of Section \ref{bstsec} for an expanded version of this remark).  
  An application of our new Banach-Stone type theorem 
is given later to the characterization of {\em symmetric} real positive projections.

One motivation for our work here on projections is the fact that the range of a contractive
projection $P$ on an operator algebra or  Jordan operator algebra is often a 
 Jordan operator algebra in the `new product' $P(x \circ y)$ (and with norm unchanged).     This is a 
reprise of the famous result of  Choi and Effros \cite[Theorem 3.1]{CE} that
the range of a completely positive 
(contractive)  
projection $P : B \to B$ on a $C^*$-algebra $B$, is again a $C^*$-algebra with new 
product $P(xy)$.    A quite deep theorem of Friedman and Russo, and a 
simpler  variant of this by Youngson, shows that something similar is
true if $P$ is simply contractive, or if $B$ is replaced by a  ternary ring of operators \cite{CPP,Y} (see also \cite{ES} for the
positive unital projection case).    
The analogous result for real completely positive completely contractive 
projections on operator algebras which have 
a contractive approximate identity is true (see \cite[Theorem 2.5]{BNp} or 
the proof of \cite[Corollary 4.2.9]{BLM}).  More recently, in 
\cite[Corollary 2.3]{BNj} the authors proved: 

\begin{theorem} \label{apch}  Let $A$ be a Jordan  operator algebra, and
$P : A \to A$ a completely contractive projection.
  \begin{itemize} \item  [(1)] The range of $P$ with product
$P(x \circ y)$, is  completely   isometrically Jordan isomorphic to a Jordan operator algebra.  \item  [(2)]  If $A$ is an associative 
operator algebra then the range of $P$ with product $P(x  y)$, is  completely isometrically algebra isomorphic to an
associative operator algebra.
\item  [(3)]  If $A$ is unital  (that is,  it has an identity of norm $1$)  and $P(1) = 1$ then the range of  $P$, with product
$P(x \circ y)$,
is unitally completely isometrically Jordan isomorphic to a unital  Jordan operator algebra.
  \end{itemize}
 \end{theorem}

If $P$ is an idempotent map on a Jordan algebra $A$, then by the  {\em new product} or $P$-{\em product} on $P(A)$ we mean the bilinear
map $P(x \circ y)$ for $x, y \in P(A)$.

We will prove variants of this last result for contractive projections in Sections 4 and 5.  These generalize the classical results of Choi, Effros, St{\o}rmer, 
and Friedman and Russo,  for positive projections on $C^*$-algebras or JB-algebras.
Such results are related to {\em conditional expectations} (resp.\ {\em Jordan conditional expectations}): these are contractive projections $P : A \to A$ 
onto a
subalgebra (resp.\ Jordan subalgebra) which are bimodule maps with respect to the subalgebra
(resp.\ satisfy $P(a \circ P(b)) = P(a) \circ P(b)$ for $a, b \in A$).    Under certain additional hypotheses one expects 
that real positive contractive  projections $P$ 
onto a  
subalgebra (resp.\ Jordan subalgebra)  to be such an expectation.   For $C^*$-algebras this is due to Tomiyama (see pp.\ 132--133 in \cite{Bla}).
We will prove results of this type in the course of our paper.    Some additional hypothesis on $P$ is needed for such results:
$P(ab) \neq P(a) b$ in general for a unital operator algebra $A$ and contractive unital (hence real
positive) projection  $P$ from $A$ onto a subalgebra containing $1_A$, and $a \in A, b \in P(A)$.   See \cite[Corollary 3.6]{BM05}. 
This is not even true in general
if $A$ is commutative (it is true if also $b \in A \cap A^*$ by the remark after that cited corollary; see also \cite{Bare} and Corollaries
 \ref{ipscor} and  \ref{ipscor2}
 below).  

The theory of Jordan  operator algebras in the sense of this paper was developed 
very recently, in \cite{BWj,BNj} (see also 
\cite{BWj2}, and see also the thesis \cite{ZWthes} 
of Zhenhua Wang for some additional results, complements, etc).  
The selfadjoint case, that is,  closed selfadjoint subspaces of a $C^*$-algebra which are closed under squares, 
are exactly what is known in the literature as 
{\em $JC^*$-algebras}, and these do have a large literature 
(see e.g.\ \cite{Rod, Rod2, HS} for references).

\bigskip

We now describe the structure of our paper.
  In Section 2, we prove many new 
results on real positive maps.  
Most of these results are foundational to the study of such maps, and of interest in their own right.  
 In Section 3,  we build on work of Arazy and Solel \cite{AS} to prove some Banach-Stone type theorems, characterizing  isometries
between operator algebras or Jordan operator algebras.  
  This will require some  analysis of multipliers and quasimultipliers, and of a certain behaviour in `Jordan multiplier algebras' which allows
us to tap into the multipliers of a containing $C^*$-algebra.  This access to the $C^*$-algebra setting is 
not obvious.   All of this is  again of interest in its own right,
but will also  be needed later in the paper,   for example for understanding symmetric projections.

In Section 4  we  prove many results on real positive 
 projections on operator  algebras or 
Jordan operator  algebras.   More particularly, we give the  variants in our setting of 
the
results from Section 2 of  \cite{BNp}, however our maps are usually no longer completely contractive, and our spaces 
are usually  Jordan operator algebras.   For example we show in and around  Theorem \ref{tr2} how to 
reduce certain questions about projections to the unital case
(see also Theorem \ref{cepro} and
Proposition \ref{reau}).    We also give variants  of the Choi-Effros result mentioned above,
in the case of a  real positive contractive projection $P$ on operator  algebras or 
Jordan operator  algebras,
and an application of this to conditional expectations.   Some of our results use as a hypothesis that 
the kernel of $P$, or at least a part of it, is generated  by the real positive elements which it contains, a condition that is always satisfied for
positive projections on $C^*$-algebras or JB-algebras.
We remark that an operator algebra  (resp.\ Jordan operator algebra)  is generated (or even is densely spanned)
 by the real positive elements which it contains if and only if it is approximately unital.  These terms are defined
below.    See for example   \cite[Proposition 2.4 and Theorem 2.1]{BRord} (resp.\ \cite[Proposition 4.4 and Theorem 4.1]{BWj}).

In Sections 5 and 6  we study  contractive projections $P : A \to A$ which are  symmetric  (that is, 
$\| I - 2P \| \leq 1$) and bicontractive  (that is, 
$\| I - P \| \leq 1$).   The main result in Section 5 is Theorem \ref{trivch2c}, which elucidates the structure of 
symmetric projections and their ranges.   This relies on our Banach-Stone type theorem from Section 3.
Again one gets that our projections are conditional expectations in this setting.
In Section 6 we study the `bicontractive projection problem'. 
We cannot expect a full solution 
here (as opposed to what was obtained in Section 5 for symmetric projections),  counterexamples were given in \cite{BNp} even if the 
projection is {\em completely bicontractive}.   As in \cite{BNp},
 we believe that the correct formulation of the 
bicontractive projection problem in our category is: 
When is the range of a bicontractive projection a (Jordan) subalgebra of $A$?
We will give a natural condition under which 
the `bicontractive projection problem' has a positive solution.

As we said, in this paper  as opposed to \cite{BNp}
 our maps are  no longer completely contractive, and our spaces 
are usually  Jordan operator algebras.   Thus we will often get weaker results; 
for example we do not have a very general version of Theorem \ref{apch} for
non-completely contractive projections.
It is worth saying though that  if one assumes the operator space setting (i.e.\ our maps are completely contractive, real completely positive, completely symmetric, completely bicontractive,
etc), then all of the results of \cite{BNp} seem to extend to  Jordan operator algebras.   We will illustrate this
with some results at the end of Section 4.  The main thing one needs to know for some of these proofs is that 
the  injective envelope 
$I(A)$ of an approximately unital Jordan operator algebra $A$ is a unital $C^*$-algebra (see Lemma \ref{ienv}  below).
For definitions and basic facts about  the injective envelope see \cite[Chapter 4]{BLM} or \cite{Ham,Pau}.

We now  give some background
and notation.  The underlying scalar field is always, 
$\Cdb$, and all maps or operators in this paper
are $\Cdb$-linear.   For background  on operator spaces and associative operator algebras we refer the reader
to e.g.\  \cite{BLM,Pau,Arv}, and for $C^*$-algebras  the reader could consult e.g.\ 
\cite{P}.   For the theory of Jordan operator algebras the reader will also want to consult \cite{BWj,BNj} frequently
 for background, notation, etc, and will often be referred 
to those papers for various results that are used here.     See also   \cite{ZWthes}.

 The letters $H, K$ are reserved for Hilbert spaces.
     A  (possibly nonassociative) normed algebra $A$  is {\em unital} if it has an identity $1$ of norm $1$, and a map $T$ 
is unital if $T(1) = 1$.    We say that $X$ is a {\em unital-subspace} (resp.\ {\em unital-subalgebra}) of
 a unital algebra $A$ if it is a subspace (resp.\ subalgebra) and $1_A \in X$.  We write $X_+$ for the positive operators (in the usual sense) that happen to
belong to $X$.    
 We write Re$(a)$ for $(a + a^*)/2,$ and Re$\, (X) = \{ {\rm Re}(a) : a \in X \}$.   This is well-defined independently of the 
Hilbert space representation of $X$ if $X$ is a unital operator space or approximately unital 
Jordan operator algebra.   One way to see this: we may assume that $A$ is unital by 
taking the bidual (discussed shortly), and in that case one may appeal to the well known 
result of Arveson concerning $A + A^*$ (see e.g.\ \cite[Lemma 1.3.6]{BLM}).
  
For a nonempty subset $S$ of a Jordan operator algebra we define joa$(S)$ to be the smallest  closed Jordan subalgebra containing $S$.  
Similarly for a
nonempty subset $S$  of an operator algebra, oa$(S)$ is the smallest  closed subalgebra containing $S$. 
A {\em Jordan homomorphism} $T : A \to B$ between
Jordan algebras
is of course  a linear map satisfying $T(a \circ b) = T(a) \circ T(b)$ for $a, b \in A$, or equivalently,
that $T(a^2) = T(a)^2$ for all $a \in A$ (the equivalence follows by applying $T$ to $(a+b)^2$). 
If  $A$ is a Jordan operator subalgebra  of $B(H)$, then the {\em diagonal}
$\Delta(A) = A \cap A^*$  is a $JC^*$-algebra.   If $A$ is
unital then as a  $JC^*$-algebra $\Delta(A)$ 
  is independent of the Hilbert space $H$  (see the third paragraph of \cite[Section 1.3]{BWj}).   An element $q$ in a Jordan operator algebra $A$
 is called  a {\em projection} if $q^2 = q$ and $\| q \| = 1$
(so these are just the orthogonal projections on the 
Hilbert space $A$ acts on, which are in $A$).    Clearly $q \in \Delta(A)$. 
Thus there is an ambiguity in notation that will pervade the paper: the 
reader should hopefully have no problem determining which sense of projection
is being used (for example, we use lower case letters for orthogonal projections
and upper case for idempotent maps).   
We say that a projection $P  : A \to A$ is a {\em Jordan conditional expectation} if $P(a \circ b) = P(a) \circ b$ for all 
$a \in A, b \in P(A)$, where $\circ$ is the Jordan product.   Note that this implies that $P(A)$ is a Jordan algebra for the $P$-product.
There are many  papers in the functional analysis literature related to contractive projections in various settings, in addition to 
the ones already cited we mention e.g.\ \cite{LL, N2, NR03}.

We recall the main facts about morphisms of $JC^*$-algebras.  Most of these are 
related to Banach-Stone type theorems.   A  Jordan $*$-homomorphism between  $JC^*$-algebras
is contractive, and if it is one-to-one then it is
 isometric.    A linear   bijection between  $JC^*$-algebras is an isometry  if and only if it is  preserves the 
triple product $x y^* x$, and if and only if it is  preserves `cubes' $x x^* x$.
These results are due to Harris \cite{Har1,Har2} in the more general case
of  $J^*$-algebras, that is, norm closed subspaces of $B(K,H)$ that are closed under the triple product
$\{ x , y, z \} = \frac{1}{2} (xy^* z + z y^* x)$.   A linear   bijection between  $JC^*$-algebras  
is a Jordan $*$-isomorphism if and only if it is approximately unital  (that is takes a Jordan contractive approximate identity to a Jordan contractive approximate identity), 
 and  if and only if it is positive.   The difficult direction of these last two iff's 
follows from facts close to the start of the present paragraph,  and taking biduals.    Also if the bidual surjective isometry is
positive then it takes $1$ to a positive unitary, that is to $1$.   Finally a contractive Jordan homomorphism $\pi$ between  $JC^*$-algebras
is a Jordan $*$-homomorphism.   To see this, restrict $\pi$ to the subalgebra generated by a selfadjoint element $x$, and then use the
well known $C^*$-algebraic version of the same result (namely that a  contractive  homomorphism between  $C^*$-algebras is a
$*$-homomorphism).   This  shows that $\pi(x)$ is selfadjoint, which yields the assertion.

We refer the reader to  \cite{Rod, Rod2, HS} for the basic theory of Jordan algebras.
In our bibliography we have also listed several other papers related to Jordan algebras and triples relevant to topics in our paper
\cite{BFT, BT, IR, M, Mc, N, RPa, SOJ, Up, Up87}.
There are certain basic formulae that hold in a Jordan operator algebra $A$  that we will use very often.   For example, 
\begin{equation}  \label{aba2}
a \circ b = \frac{1}{2} [(a+b)^2 - a^2 - b^2] , \qquad a, b \in A . 
\end{equation}
and
\begin{equation}  \label{aba}
aba = 2 (a \circ b) \circ a -  a^2 \circ b , \qquad a, b \in A . 
\end{equation}
Also,  $\frac{1}{2} (abc + cba ) = (a\circ b)\circ c + (b\circ c)\circ a - (a \circ c) \circ b$. Thus $abc + cba \in A$ for $a, b, c \in A$.   
See e.g.\ \cite[p.\ 25]{HS}. 
  If $p$ is a projection in $A$ then  $p \circ a = \frac{1}{2} (a +  pap  - p^{\perp} a p^{\perp})$.

A projection $q$ in a Jordan operator algebra $A$ will be called
{\em central} if $qxq = q \circ x$ for all $x \in A$.  This is equivalent to
$qx = xq$ in any $C^*$-algebra containing $A$ as a
Jordan subalgebra, by the first labeled equation in \cite{BWj}.     It is also equivalent to  $q$ being central in any generated
(associative) operator algebra, or in a generated $C^*$-algebra.   This notion is independent of the particular generated
(associative) operator algebra since it is captured by the intrinsic formula $qxq = q \circ x$ for $x \in A$.  

 A {\em Jordan ideal}  of  a Jordan algebra $A$ 
is a subspace $E$ with  $\eta \circ \xi \in E$
for $\eta \in E, \xi \in A$.

If $A$ is a Jordan subalgebra of a $C^*$-algebra $B$ then $A^{**}$ with its Arens product 
is a Jordan subalgebra of the von Neumann algebra $B^{**}$ (see \cite[Section 1]{BWj}).   Since 
the diagonal $\Delta(A^{**})$ is 
a $JW^*$-algebra  (that is a weak* closed $JC^*$-algebra), it follows that $A^{**}$ is closed under meets and joins of projections.
If  $P : A \to A$ a  contractive projection  on a Jordan operator algebra then
$Q = P^{**} : A^{**} \to A^{**}$  is a  contractive projection and 
$Q(A^{**})$ is the weak* closure of $P(A)$.  Indeed if $P(x_t) \to \eta \in A^{**}$ weak* then 
$P(x_t) \to Q(\eta)$ weak*.  So  the weak* closure of $P(A)$ is contained in $Q(A^{**})$.  Conversely,
$a_t \to \eta \in A^{**}$ weak* implies $P(a_t) \to Q(\eta)$ weak*, so that  $Q(A^{**})$ is contained in
 the weak* closure of $P(A)$.

A {\em Jordan  contractive approximate identity}
 (or {\em Jordan cai} for short) 
for $A$  is a net
$(e_t)$  of contractions with $e_t \circ a \to a$ for all $a \in A$.
A {\em  partial cai} for $A$ is a net consisting of real positive  elements in $A$,
that acts as a cai  (that is, a contractive approximate identity)
for the ordinary product in
every $C^*$-algebra  which contains and is generated by $A$ as a closed Jordan subalgebra.    If a 
  partial cai for $A$ exists then $A$ is called {\em approximately unital}.  
It is shown in \cite[Section 2.4]{BWj}   that if $A$ has a Jordan cai then it has a partial cai.  
Indeed any net converging weak* to $1_{A^{**}}$ may be modified as
in the proof of  \cite[Lemma 2.6]{BWj} to yield a partial cai for $A$.

We recall that 
every Jordan operator algebra $A$ has a  unitization $A^1$ 
which is unique up
to isometric Jordan homomorphism (see
 \cite[Section 2.2]{BWj}).  A {\em state} of an approximately unital Jordan 
operator algebra $A$ is a functional with $\Vert \varphi \Vert = \lim_t \, \varphi(e_t) = 1$
for some (or every) Jordan cai $(e_t)$ for $A$.  These extend to states of the unitization $A^1$.  They also
extend to a state (in the $C^*$-algebraic sense) on any $C^*$-algebra $B$ generated by $A$, and conversely
any state on $B$ restricts to a state of $A$.  See \cite[Section 2.7]{BWj}
for details.

Suppose that $E$ is a closed subspace of $B(H,K)$, and that $u$ is a contraction in $B(H,K)$ with $bu^* b \in E$ for all $b \in E$.
Define $E(u)$ to be $E$  equipped with  Jordan product $(bu^* c + c u^* b)/2$.  Then 
$E(u)$ is completely isometrically  Jordan  isomorphic to a  Jordan operator algebra.
This follows from the proof of \cite[Theorem 2.1]{BNj} (which also shows that every Jordan operator algebra arises in this
way).   We will need, and will reprove now, a simple case of this: if in addition $u$ is unitary in $B(H,K)$,
then $Eu^*$ is a Jordan subalgebra of $B(K)$, and right multiplication $R_{u^*}$ by $u^*$ is a 
completely isometric  Jordan   homomorphism from $E(u)$ onto $E u^*$.

Because of the uniqueness of unitization up to isometric isomorphism, for a Jordan operator algebra $A$ 
we can define unambiguously ${\mathfrak F}_A = \{ a \in A : \Vert 1 - a \Vert \leq 1 \}$.  Then 
 $\frac{1}{2} {\mathfrak F}_A = \{ a \in A : \Vert 1 - 2 a \Vert \leq 1 \} \subset {\rm Ball}(A)$.
Note that $x \in {\mathfrak F}_A$ if and only if 
$x^* x \leq x + x^*$.    Similarly, ${\mathfrak r}_A$, 
the {\em real positive} or {\em accretive} elements in $A$, may be defined as the 
the set of  $h \in A$ with  Re $\varphi(h) \geq 0$ for all states $\varphi$ of $A^1$.
This is equivalent to all the other usual conditions 
characterizing accretive elements as we said in \cite[Section 2.2]{BWj}.  
Note that the `negatives' of the accretive elements are sometimes called {\em dissipative} (see e.g.\ \cite[Definition 2.1.8]{Rod}).
We have for example  ${\mathfrak r}_A = \{ a \in A : a + a^* \geq 0 \}$, where 
the adjoint and sum here is in (any) $C^*$-algebra  containing $A$ as a
Jordan subalgebra.   We also have
${\mathfrak r}_A = \overline{ \Rdb_+ {\mathfrak F}_A}$.   If $A$ is a Jordan subalgebra of a Jordan operator algebra $B$
then ${\mathfrak F}_A = {\mathfrak F}_{B} \cap A$
and ${\mathfrak r}_A = {\mathfrak r}_{B} \cap A$.    

A linear map $T : A \to B$ between Jordan operator algebras is 
{\em real positive} if $T({\mathfrak r}_A)
\subset {\mathfrak r}_B$.     The real positive maps on $JC^*$-algebras are just  the positive maps.  
This follows, after  taking the bidual, from the fact that real positive maps on operator systems are just  the positive maps.  
In turn, the harder direction of the last fact follows e.g.\ as in the proof of  \cite[Theorem 2.4]{BBS}. 

  The Jordan 
multiplier algebra $JM(A)$ of an approximately unital Jordan
operator algebra $A$ is
$$JM(A) = \{ x \in A^{**} : x \circ A \subset A \} .$$  
This was defined in \cite[Definition 2.25]{BWj} but not used, and
 it was not proved there that $JM(A)$ is a Jordan operator algebra (or  rather a misleading hint was given for this).

\begin{lemma} \label{jmisjoa} The Jordan 
multiplier algebra $JM(A)$ of an approximately unital Jordan
operator algebra $A$ is  a Jordan operator algebra.
\end{lemma}

\begin{proof}
In fact  this may be proved using some consequences of the Jordan identity $$(x^2 \circ b) \circ x = x^2 \circ (b \circ x),$$ following
what seems a well known path valid in Jordan algebras \cite{Rod}.   The first author  together with Z.\ Wang checked this route following
the ideas in  \cite{Rod}: 
Dropping the $\circ$ notation, rewrite this as $(x^2 b) x = x^2(bx)$, and replace $x$ by $ta + c$ for scalar $t$.   Then expand the parentheses, and equate coefficients
of $t$.   Writing $[a,b,c] = (ab)c - a(bc)$, we have proved that $2[a, b, ac]+[c, b, a^2]=0.$ Letting $c=b,$ we have  $2[a, b, ab]+[b, b, a^2]=0,$ which yields $$b^2a^2=b(ba^2)+2a(b(ab))-2(ab)^2.$$   This works in any Jordan algebra, and in particular in $A^{**}$.  In particular it holds
if $b \in JM(A)$ and $a \in A$, if $A$ is an  approximately unital Jordan 
operator algebra.   Next we note that in the latter case squares linearly span $A$ (using the fact that $A = {\mathfrak r}_A -  {\mathfrak r}_A$, and 
real positive elements have roots, 
or by using (\ref{aba2}) for example, with $a$ replaced by the cai).
It follows that $b^2 \circ A \subset A$ for $b \in JM(A)$, as desired.  
  \end{proof} 
 
  It also follows that $b a b =  2 (b \circ a) \circ b -  b^2 \circ a \in A$ if $b \in JM(A), a \in A$.

If $A$ is in addition an (associative) operator algebra then $JM(A) = M(A)$.  Indeed clearly $M(A) \subset JM(A)$.
For  the converse, if  $x \in JM(A)$ and  $b \in A$  then we have
$x b^2 = 2(x \circ b) b - bxb \in A$.   Similarly $b^2 x \in A$.  
If $A$ is approximately unital then  squares span $A$, 
as we said in the
last paragraph, so that $x \in M(A)$.  

We define a   {\em hereditary subalgebra}  of a Jordan operator algebra $A$, or {\em HSA} of $A$ for short, to be a Jordan subalgebra 
$D$ possessing a Jordan cai (this was defined above), which satisfies
$aAa \subset D$  for any $a \in D$ (or equivalently, by replacing $a$ by $a + c$, such that
if $a,c \in D$
and $b \in A$ then $abc  + cba \in D$).   We say that a projection in $A^{**}$ is {\em open} in $A^{**}$
if there is a net $(x_t)$ in $A$ with
$$x_t = p x_t p \to p \; \; {\rm weak}^* .$$    This is a variant of the open projections for $C^*$-algebras 
in the sense of Akemann (see e.g.\ \cite{Ake2}).  The ensuing noncommutative topology for possibly nonselfadjoint  operator algebras 
have been worked out in a series of papers by the first author with Read, the second author, Hay, Wang and others (see our bibliography).
See \cite{BWj,BNj} for the Jordan operator algebra case.
  If $p$ is open in $A^{**}$ then   $D=pA^{**}p\cap A =\{a\in A: a=pap \}$ is a  hereditary subalgebra (HSA) of $A$, and the Jordan subalgebra $D^{\perp \perp}$ of $A^{**}$ has identity $p$,
and.   Conversely, every hereditary subalgebra of $A$ is of this form, and we say that $p$ is the {\em support projection} of $D$.  Indeed $p$ is the weak* limit of any Jordan cai from the HSA.

\section{New results on real positive maps}

\begin{lemma} \label{nrp} If $T : A \to B$ is a  real positive linear map  between unital  (resp.\ approximately unital)  Jordan operator algebras
then  $T$ 
 is bounded and  $\| T \| = \| T(1) \|$  (resp.\ $\| T \| = \| T^{**} \| =  \| T^{**}(1) \| = \sup_t \, \| T(e_t) \|$, 
if $(e_t)$ is a Jordan cai for $A$).   \end{lemma}
\begin{proof}   
By \cite[Corollary 4.9]{BWj}, $T$ 
 is bounded and extends uniquely to a positive
$\tilde{T} : A + A^* \to B + B^*$.    If $A$ is unital then  the proof of  \cite[Corollary 2.8]{Pau} (but replacing $A$ there with a Jordan subalgebra)
 yields $\| T \| = \| T(1) \|$.
In the approximately unital case  $T^{**} : A^{**} \to B^{**}$  is  real positive
using \cite[Theorem 2.8]{BWj} if necessary. 
By the unital case above we have  $\| T \| = \| T^{**} \| =  \| T^{**}(1) \| = \sup_t \, \| T(e_t) \|$ (the latter because the 
norm is semicontinuous for the weak* topology,
and $e_t \to 1$ weak* by \cite[Lemma 2.6]{BWj}).  \end{proof}

A unital functional on a unital operator space (resp.\ operator system) is contractive if and only if it is real positive (resp.\ positive).
Such maps are well known to be real completely positive (resp.\ completely positive)--see e.g.\
\cite[Remark 4.10]{BWj}.   A unital linear contraction on a unital operator space (resp.\ operator system) is real positive (resp.\ positive)
by e.g.\ the same  Remark from \cite{BWj}. However the converse is false in general. 

\begin{corollary} \label{trivj} Let $A, B$ be  approximately unital Jordan operator algebras,
and let  $T : A \to B$ be a contraction which is approximately unital (that is, takes some Jordan cai to a Jordan cai), or more generally for which 
$T^{**}$ is unital.   
Then $T$ is real positive.

If $\theta : A \to B$ is a contractive Jordan homomorphism then $\theta$ is real positive.
  \end{corollary}

\begin{proof}    By taking the second dual we may assume that $A, B$ are unital, and that $T(1) = 1$.    
Then the first assertion follows from the lines before the corollary.   The second follows easily from the first after replacing 
$B$ with $\overline{\theta(A)}$.
\end{proof}

The following gives a very useful way to `reduce to the unital case'.  It is a generalization of the fact that positive maps on 
$C^*$-algebras extend to the unitization \cite{Pau}.

\begin{theorem} \label{cepro}   Let $A$ and $B$ be approximately unital Jordan operator algebras, and write
$A^1$ for a Jordan operator algebra unitization of $A$ with $A \neq A^1$.  Let $C$  be a unital Jordan operator
algebra containing B as a closed subalgebra.
\begin{itemize}  \item [(1)]
A  real positive contractive  linear map  $T : A \to B$ extends to a unital real
 positive contractive linear map from $A^1$ to $C$.   
 \item [(2)] A real positive contractive projection
on $A$  extends to a unital real
 positive contractive projection on $A^1$.  
\end{itemize}
\end{theorem}  \begin{proof}   Clearly (2) follows from (1).   Let 
$\tilde{T} : A^1 \to C$ be the canonical unital extension of $T$, and write $e, f$ for the units of $A^1$ and $C$ (so $\tilde{T}(e) = f$).    
Suppose that $A$ is a Jordan subalgebra of a $C^*$-algebra $D$.  We may assume by \cite[Corollary 2.5]{BWj} that $e = 1_{D^1} \notin D$. 
Suppose that Re $(x + \lambda e) \geq 0$ for  $x \in A$ and  scalar $\lambda$.    We need to prove that 
Re $(T(x)  + \lambda f) \geq 0$.   This is clear if Re $(\lambda)  = 0$, so suppose the contrary.  
Now 
Re $(\lambda) > 0$ (by considering the character 
$\chi$ on $D^1$ that annihilates  $D$; this
is a state so that Re$(\chi(x) + \lambda) = {\rm Re}(\lambda) \geq 0$).   
 Since Re$(x + \lambda e) \geq 0$ we have $-\frac{1}{{\rm Re} (\lambda)} \, {\rm Re} (x) \leq e$.   
Let $$x_n = -   \frac{n-1}{n \, {\rm Re} (\lambda)} \, x \, , \; \; \; 
y = {\rm Re} (x_n) \leq   \frac{n-1}{n} \, e$$ and $z =y _+ 
\leq  \frac{n-1}{n} \, e$.   By \cite[Theorem 4.1 (2')]{BWj} there exists 
a contraction $a \in A$ with $0 \leq z \leq {\rm Re} (a) \leq e$.   Now 
%$x_n = a - (a - x_n)$, and
 Re $(a - x_n)   \geq 0$, since  Re $(x_n) = y \leq y_+ = z  \leq {\rm Re} (a)$.
Also 
 $\| {\rm Re} \, (T(a)) \| \leq 1$  since $a$ and $T$ are contractions, and therefore  $0 \leq {\rm Re} \, (T(a)) \leq f$.  Also, Re $T(a - x_n) \geq 0$, 
so that  Re $(T(x_n)) \leq {\rm Re} \, (T(a)) \leq f$.   That is, 
$$ -  \frac{n-1}{n \, {\rm Re} (\lambda)} \, {\rm Re} \, (T(x))  \leq f.$$
Letting $n \to \infty$ we have that ${\rm Re} \, (T(x) +  \lambda f) \geq 0$ as desired.  
Hence $ \tilde{T}$ is  a unital real
 positive  map, and thus is  contractive by  Lemma \ref{nrp}.  
 \end{proof}

Of course the extensions in the previous result are unique.  

\begin{lemma} \label{nrp2}   A real positive linear functional on a unital operator subspace or  approximately unital 
Jordan subalgebra of a  $C^*$-algebra
$B$, extends to a positive functional on $B$ with the same norm.  \end{lemma}
\begin{proof}    By taking the bidual we may assume that the domain is a unital  operator subspace $A$.  (We may then assume if we wish that that $B = 
C^*(A)$.).   One way to finish is then to note that  the functional is  real completely positive by the Remark 4.10 in \cite{BWj},
and appeal to \cite[Theorem 2.6]{BBS}.   (Other approaches: we may assume that   the functional is  unital by Theorem \ref{cepro}.
Or use the fact from e.g.\ \cite[Proposition 3.1]{BWj2} and/or \cite[Lemma 2.9.15]{Rod} that every real positive linear functional in these situations is a non-negative 
multiple of a state, and the fact that states extend.)
\end{proof}

The following result is useful 
for questions about real positivity because it  shows that we can 
often get away with 
 working with 
the simpler set ${\mathfrak F}_A = \{ x \in A : \| 1 - x \| \leq 1 \}$, instead of the more complicated set 
of real positive or accretive elements.    Indeed the condition $\| 1 - x \| \leq 1$ is a lot stronger than
the condition $x + x^* \geq 0$.

\begin{theorem} \label{ocp}  
  A linear map $T : A \to B$ between approximately unital Jordan operator algebras 
is  real positive and contractive if and only if $T({\mathfrak F}_A) \subset {\mathfrak F}_B$. 
\end{theorem}  \begin{proof} Any unital contraction $T : A \to B$ between unital Jordan operator algebras (or unital 
operator spaces) satisfies  $T({\mathfrak F}_A) \subset {\mathfrak F}_B$.
Indeed if $\| 1 - x \| \leq 1$ then $\| 1 - T(x) \| = \| T(1-x) \| \leq 1$.  

A real positive contraction $T : A \to B$ between approximately unital  Jordan operator algebras extends  by Theorem \ref{cepro}     to a real positive unital 
contraction $\tilde{T} : A^1 \to B^1$.  Then  $T({\mathfrak F}_A) \subset \tilde{T}({\mathfrak F}_{A^1}) \subset {\mathfrak F}_{B^1}$, by the last paragraph.
Hence $T({\mathfrak F}_A) \subset   B \cap {\mathfrak F}_{B^1} = {\mathfrak F}_B$.    

Conversely, if $T({\mathfrak F}_A) \subset {\mathfrak F}_B$ then $T$ is  real positive since ${\mathfrak r}_A = \overline{\Rdb^+ {\mathfrak F}_A}$.   
Also $T^{**}({\mathfrak F}_{A^{**}}) \subset {\mathfrak F}_{B^{**}}$ 
by \cite[Theorem 2.8]{BWj}.   Therefore $\| 1 - 2 T^{**}(1) \| \leq 1$ since
$1 \in \frac{1}{2} {\mathfrak F}_A$, so that $\| T^{**}(1) \| \leq 1$.   Hence $T$ is a contraction by Lemma \ref{nrp}.  
 \end{proof}  

The following shows how every weak* continuous real positive contraction gives rise to a real positive contractive projection
with range the fixed points of the contraction:

\begin{lemma} \label{fpes}   Let $\Phi : M \to M$ be a weak* continuous real positive contraction  on a unital weak* closed Jordan operator 
algebra, and let $M^\Phi$ be the set of  fixed points of $\Phi$.   Then there exists a  real positive contractive projection on $M$ with range $M^\Phi$. 
\end{lemma}

\begin{proof}  
One may follow the proof in \cite[Corollary 1.6]{ES}, taking weak* limits in  the unit ball 
of $B(M,M) = (M \hat{\otimes} M_*)^*$ of averages of powers of $\Phi$.
\end{proof}  

 If in addition to the hypotheses of the last result $\Phi$ is real completely positive then  $M^\Phi$ is a unital Jordan operator 
algebra with respect to the new product coming from the projection by e.g.\ \cite[Theorem 2.5]{BNp}.    Conversely any unital Jordan operator algebra is the set of fixed points of a real completely positive unital contraction (even tautologically).   Sometimes the selfadjoint analogue of $M^\Phi$ is called 
the {\em Poisson boundary}.

\begin{lemma} \label{ispos}  Let $T : A \to B$ be a real positive 
map between  Jordan operator algebras.  Then $T(\Delta(A)) \subset \Delta(B)$, and $T$ restricts to a positive linear map 
from $\Delta(A)$ to $\Delta(B)$.   Thus  $0 \leq T(1) \leq 1$ if $A$ is unital and $T$ is contractive.
\end{lemma}

\begin{proof}    
This follows as in the proof of  \cite[Lemma 2.3]{BNp}, but 
using the fact that $\Delta(A)$ is a $JC^*$-algebra and operator system.  Indeed regarding $B$ as a closed
Jordan subalgebra of $B(H)$, and regarding the restriction of $T$ as
a real positive map on the $JC^*$-algebra $\Delta(A)$,
we have that it is positive, hence $*$-linear.  So $T(\Delta(A)) \subset \Delta(B)$, and the rest is clear.
\end{proof}   

\begin{lemma} \label{ishe2}  Let $A$ be a unital-subspace of  a unital $C^*$-algebra  $B$, and let 
$q \in {\rm Ball}(B)_+$ with $q \circ A \subset A$ and  $q^2 \circ A \subset A$.
Suppose that  $T : A \to B(H)$ is a real positive map on  $A$,
and $T(1) = T(q)$.   Then $T(a) = T(qaq) = T(q \circ a)$ for $a \in A$.  \end{lemma}
\begin{proof}    
  By (\ref{aba}) we have  $qaq = 2 (q \circ a) \circ q - q^2 \circ a \in A$ for $a \in A$.   
 Let $\psi$ be a state on $B(H)$, and set $\varphi
= \psi \circ T$.   
This  is real positive on $A$,
 hence  by Lemma \ref{nrp2}   it  extends to a positive functional $\tilde{\varphi}$ on $B$.   
   Thus for  $a \in {\rm Ball}(A)$,  by the Cauchy-Schwarz inequality for states of a $C^*$-algebra there exists a constant
$K$ such that 
$$|\varphi( (1-q)a(1-q))|^2 = | \tilde{\varphi}(b(1-q)^{\frac{1}{2}})|^2
\leq  K \, \tilde{\varphi}(1-q) = \, \psi(T(1) - T(q)) = 0
,$$    where $b =  (1-q)a(1-q)^{\frac{1}{2}}$.  
 Similarly, $| \tilde{\varphi}(q a (1-q))|^2$ and $| \tilde{\varphi}((1-q) a q)|^2$ are zero. 
 Hence the numerical radii of $T((1-q) a (1-q))$ and $T((1-q) aq + qa(1-q))$ are $0$,  and so $T((1-q) a (1-q)) = 
T((1-q) aq + qa(1-q)) = 0$.  We deduce that 
$$T(a) = T((1-q) a + qa) =  T((1-q) a(1-q) + (1-q)a q  + qa(1-q) + qaq) =  T(qaq)$$ for all $a \in A$.    
To obtain the last equality, apply the equality before it with $a$ replaced by $q \circ a$.
\end{proof}  

{\bf Remark.}   A similar but easier proof shows:   Let $A$ be a unital-subspace of a unital $C^*$-algebra  $B$ and let 
$q \in {\rm Ball}(B)_+$ with $qA + Aq \subset A$.
Suppose that  $T : A \to B(H)$ is real positive on  $A$,
and $T(1) = T(q)$.   Then $T(a) = T(qa) = T(aq) = T(qaq)$ for $a \in A$. 

We also are indebted to the referee for a correction to the hypotheses of the last result.  
  
\medskip

The following very similar result is proved in \cite{Brp}, where it is used to characterize real positive projections on operator algebras
taking values in a selfadjoint subspace.

\begin{lemma} \label{ishe3} Let $A$ be  a unital 
operator space (resp.\ approximately  unital Jordan
operator algebra),
and let  $T : A \to B(H)$ be a unital  (resp.\ real positive) contraction.
Suppose that $e$ is a projection in $A$ with $e \circ A \subset A$,
such that $q = T(e)$ is a projection in $B(H)$.   Then
$T(eae) = q T(a) q$ 
and  $T(a \circ e) =T(a) \circ q$ for all $a \in A$.    
 \end{lemma}

\begin{lemma} \label{paauj}  If $T : A \to B$ is a  real positive map  between approximately unital Jordan operator algebras
then ${\rm joa}(T(A))$  is an approximately unital Jordan operator algebra. \end{lemma}
\begin{proof}    Let $S = T({\mathfrak r}_A)
\subset {\mathfrak r}_B$.  Since $A = {\mathfrak r}_A - {\mathfrak r}_A$ by \cite[Theorem 4.1]{BWj}, 
${\rm joa}(T(A)) = 
{\rm joa}(S)$   is an approximately unital Jordan operator algebra by \cite[Proposition 4.4]{BWj}.   \end{proof}

The following is a nonselfadjoint analogue of the well known fact that 
the positive part of the kernel of a
completely positive map $T$ on a $C^*$-algebra $B$ has the following `ideal-like' property 
$$T(xy)^* T(xy) \leq K \, T(y^{\frac{1}{2}} y^{\frac{1}{2}} x^* x y^{\frac{1}{2}} y^{\frac{1}{2}}) 
\leq K'  \, T(y) = 0 , \qquad y \in {\rm Ker}(T)_+ , x \in B,$$
using the Kadison-Schwarz inequality.   Here $K, K'$ are constants depending on $T$, and $K'$ depends also on $x$.
So $T(xy) = 0$.   Similarly $T(yx) = 0$.  In fact this is also true for                                               
positive maps on $JC^*$-algebras, as follows e.g.\ from the next result
(using the fact that positive maps on   $JC^*$-algebras or operator systems are obviously 
real positive).     Note that the entire kernel is rarely an ideal (consider
for example the map consisting of integration on $L^\infty([-1,1])$).   

\begin{lemma} \label{kerlem} Suppose that  $A$ is an approximately unital Jordan operator algebra (resp.\ operator algebra)
and that $T : A \to B(H)$ is a real positive map on $A$.  If $x \in A$ and $y \in {\mathfrak r}_{A} \cap {\rm Ker}(T)$  then $x \circ y \in {\rm Ker}(T)$ 
and $yxy  \in {\rm Ker}(T)$ (resp.\ $xy$ and $yx$ are in ${\rm Ker}(T)$).    \end{lemma} 
\begin{proof}   Let $\varphi$ be a state on $B(H)$,  which is real positive.   As in the proof of Lemma \ref{ishe2}, 
$\varphi \circ T$ extends to a positive functional 
$\psi$ on $C^*(A)$.  So in the operator algebra case, by the Cauchy-Schwarz inequality
for states on a $C^*$-algebra,  $$|\psi(xy)|^2 \leq
K \, \psi(y^* y) \leq K \, \psi(y + y^*) = 0, \qquad  x \in A, y \in \mathcal{F}_{A} \cap {\rm Ker}(T) .$$
Here $K$ is a constant depending on  $x$.
We used the fact that $y \in {\mathfrak F}_A$ if and only if   
$y^* y \leq y + y^*$.  Thus $\psi(xy) = \varphi(T(xy)) = 0$.   In the Jordan case a similar argument gives
$\psi(yx) = 0$, so $$\psi(xy+yx) = \varphi(T(xy + yx)) = 0.$$    Hence $T(xy + yx) = 0$
(resp.\ $T(xy) = 0$), since states  
on  $B(H)$ 
separate points.  So $x \circ y \in {\rm Ker}(T)$, and 
similarly $yxy  \in {\rm Ker}(T)$.    Finally use the fact that ${\mathfrak r}_A = \overline{\Rdb^+ {\mathfrak F}_A}$
to replace ${\mathfrak F}_A$ by ${\mathfrak r}_A$. 
\end{proof} 

{\bf Remark.}   The referee of our paper has generalized the technique seen in our proof of Lemmas \ref{kerlem} and \ref{ishe2}
to obtain the following technical principle which in turn
may be used to unify the proofs of these lemmas:  Let $A$ be a unital operator subspace or Jordan subalgebra of a unital $C^*$-algebra
$B$.
Suppose that every  real positive functional on $A$ extends to a positive functional on $B$ (as is always  the case if 
$A$ is unital or approximately unital by Lemma \ref{nrp2}).    Suppose that $T : A \to B(H)$ is a real positive linear map
and that $y \in {\mathfrak F}_B \cap {\rm Ker}(T)$.  Let $F(x) = \sum_{k=1}^n \, a_k x b_k$ for $n \in \Ndb$ and $a_k, b_k \in B$.
Assume that either $a_k$ or $b_k$ is $y$ for all $k$.  Then $A \cap F(B) \subset  {\rm Ker}(T)$.

To prove this principle, simply  follow the proof of Lemma \ref{kerlem} almost verbatim to deduce 
that $\psi(xy) = 0$ for $x \in B$.  Similarly $\psi(yx) = 0$.  Thus 
if $F(b) \in A$ then $\psi(F(b)) = 0 = \varphi(T(F(b)))$.  As in the proof we are following this implies that $T(F(b)) = 0$.

For example, we show how to use this principle to prove  Lemma \ref{kerlem}.  As in the proof of  \ref{kerlem} we may assume  
that  $y \in 
{\mathfrak F}_B \cap {\rm Ker}(T)$.  The conclusion then follows by  the principle and 
taking $F(x)$ to be first $x \circ y$ and then $yxy$, for  $x \in A$.

\begin{lemma} \label{genf} Suppose that a closed subspace $J$ of a Jordan operator algebra $A$ is contained in the closed Jordan algebra (or even in the hereditary subalgebra) generated by 
 $C = {\mathfrak r}_J = J \cap {\mathfrak r}_A$, the set 
of real positive elements which $J$ contains.   
\begin{itemize} \item  [(1)] 
 If $x  y x \in  J$ for $y \in A, x \in C$, then $J$ is an HSA in $A$,
and $$J = \Rdb_+ ({\mathfrak F}_J - {\mathfrak F}_J) = {\rm Span}({\mathfrak F}_J) = {\rm Span}({\mathfrak r}_J).$$
If in addition $x  \circ y \in J$ for $y \in A, x \in C$, then $J$ is an approximately unital Jordan ideal in $A$.  
 \item  [(2)]  If $A$ is approximately unital, and $T : A \to B(H)$ is real positive, and if $J = {\rm Ker}(T)$
satisfies the condition before {\rm (1)}, then all the conditions in {\rm (1)} hold, so that ${\rm Ker}(T)$
is  an approximately unital Jordan ideal in $A$.  \end{itemize} 
 \end{lemma} 
\begin{proof}    
Note that $C$ is convex, so that by a formula in \cite[Theorem 3.18 (2)]{BWj}, the  hereditary subalgebra $D$ generated by $C$ is
the closure of  $\{ xAx : x \in C \}$, which by hypothesis is contained in $J$.   So $D = J$. 
Hence $J  = \Rdb_+ ({\mathfrak F}_J - {\mathfrak F}_J) = {\rm Span}({\mathfrak F}_J)$ by  \cite[Proposition 4.4 and Theorem 4.1]{BWj}.   Now the last statement
of (1)  is obvious.    

For (2), Lemma \ref{kerlem} shows that for any $a \in A, c \in J \cap {\mathfrak r}_A$, we have $ca c$ and $a \circ c$ in 
${\rm Ker}(T)$.    Now apply   (1). 
 \end{proof}

{\bf Remark.}  Note that the proof of (1) works even
if $C$ is replaced by $J \cap {\mathfrak F}_A$.  One may also prove a variant of (1) where $C$ is a convex subset of  $J \cap {\mathfrak r}_A$.

\begin{corollary} \label{genf2}  If $A$ is an approximately unital operator algebra,  $T : A \to B(H)$ is real positive, and if $J = {\rm Ker}(T)$  is contained in the HSA generated by $J \cap {\mathfrak r}_A$ then ${\rm Ker}(T)$
is  an approximately unital ideal in $A$. \end{corollary}

The following result is a fundamental fact concerning extending contractive linear maps on hereditary subalgebras (HSA's) of $A$,
and concerning the   uniqueness of such extension.     We will call the extension the {\em zeroing extension} since if $p$ is the  support projection of the HSA $D$ then this extension is 
zero on the `complement'  $\{ p^\perp a  p^\perp + p^\perp a  p+ pa p^\perp : a \in A \}$ of $D$.  
 In  \cite[Proposition 2.11]{BHN} and \cite[Corollary 3.6]{BWj} this extension was done for completely  contractive linear maps in various settings, with a  
similar proof.  We have chosen 
to briefly include most of the details of the proof for completeness and also  to exhibit the `zeroing' construction.

\begin{theorem} \label{hsaex}
	Let $D$ be a HSA in an approximately unital Jordan operator algebra $A$.
Then any contractive map $T$ from $D$ into a unital weak* closed Jordan operator algebra $N$
such that $T(e_t) \to 1_N$ weak* for some partial cai $(e_t)$ for $D$, has a unique  contractive  extension
$\tilde{T} : A \to N$ with $\tilde{T}(f_s) \to 1_N$ weak* for some (or all) partial cai $(f_s)$ for $A$.
This extension is real positive.
\end{theorem}
\begin{proof}   The canonical weak* continuous extension $\hat{T} : D^{**} \to N$ is unital 
and  contractive, and can be extended to a weak* continuous unital contraction 
$\Phi(\eta) = \hat{T}(p \eta p)$ on $A^{**}$, where $p$ is the support projection of $D$.   This is real positive,
and in turn restricts to a real positive contractive 
$\tilde{T} : A \to N$ with $\tilde{T}(f_s) \to 1_N$ weak* for all partial cai $(f_s)$ for $A$.
Now for the uniqueness.  Any other such extension $T' : A \to N$ extends to a 
weak* continuous unital contraction $\Psi : A^{**} \to N$, and $\Psi(p) = {\rm w}^*\lim_t \Phi(e_t) = 1_N$.  
Then for $\eta \in A^{**}$ we have by Lemma \ref{ishe2} that 
$$\Psi(\eta) = \Psi(p \eta p) = 
\hat{T}(p \eta p).$$
Thus $T'(a) = \Psi(a) = \tilde{T}(a)$ for $a \in A$.  
 \end{proof}

\section{Banach-Stone theorems}  \label{bstsec}

  There are very many 
Banach-Stone type  theorems in the literature.  For example it is proved in \cite[Corollary 2.8]{AS} that a unital surjective isometry between unital
Jordan operator algebras is a Jordan homomorphism.    In this section we wish to extend this latter result to nonunital 
 surjective isometries between approximately unital
Jordan operator algebras.    An attempt at this was made in  Proposition 4.15 in 
\cite{BWj}, but for complete isometries.   There is an error in the proof of that result
 in the Jordan algebra case (it is correct in the associative operator algebra case).   The correct result  for 
Jordan operator algebras  will be included in Theorem \ref{ubsj} below, although the next proposition, which is essentially due to 
Arazy and Solel \cite{AS}, is a simpler variant of it.  

As we mentioned in the introduction, there exist linearly isometric unital $JC^*$-algebras which are not Jordan $*$-isomorphic  (see e.g.\ \cite[Section 5]{BKU} and  \cite[Antitheorem 3.4.34]{Rod}).   Hence they are not even Jordan isomorphic by
\cite[Corollary 3.4.76]{Rod}. Therefore  
Banach-Stone theorems for nonunital isometries between Jordan operator algebras are not going
to look quite as one might expect: one cannot expect the Jordan isomorphism appearing in the 
conclusion to map onto the second $C^*$-algebra exactly. (As the referee pointed out
one may obtain the latter only in very particular cases \cite[Theorems 5.10.9 and 5.10.111]{Rod}, and in general the best one can say with this formulation is given by  \cite[Proposition
5.10.114]{Rod}.)  
Nonetheless we will show that with a slight adjustment of the setting, one may get a reasonable 
Banach-Stone type  theorem for nonunital isometries between Jordan operator algebras.

We define a {\em quasimultiplier} of a Jordan operator algebra $B$ to be 
an element $w \in B^{**}$ with $bw b \in B$ for all $b \in B$.  
A unitary in a unital selfadjoint Jordan operator algebra is an element $u$ with $\frac{1}{2}(u u^* 1 + 1 u^* u) = 1$; this condition
is easily seen to imply 
that $u^* u = u u^* = 1$ in any generated $C^*$-algebra.    
 
\begin{lemma}  \label{qm}  If $B$ is a  Jordan operator algebra, and $u$ is a unitary in $\Delta(B^{**})$ such that 
$u^*$ is a quasimultiplier of $B$, and if we define $B(u)$ to be $B$ equipped with the new 
Jordan product  $\frac{1}{2}(x u^* y + y u^* x)$, then
$B(u)$ is a Jordan operator algebra.  If $B$ is isometrically (resp.\
completely isometrically) a Jordan subalgebra of a $C^*$-algebra $D$
then $Bu^*$ is a Jordan subalgebra of $D^{**}$ 
 isometrically
(resp.\
completely isometrically)  Jordan  isomorphic to $B(u)$.  \end{lemma} \begin{proof}   This is evident from the discussion
about $E(u)$ in the introduction: 
 right multiplication $R_{u^*}$ by $u^*$ is the required  Jordan  isomorphism from $B(u)$ onto $B u^*$.
\end{proof}
 
{\bf Remark.}  We will use the idea in the last proof many times in the remainder of this Section.  Note that the Jordan structure of $B u^*$ is
independent of the particular containing $C^*$-algebra $D$. 

\begin{proposition} \label{bsjoa}    Suppose that $T : A \to B$ is an isometric surjection between approximately unital
Jordan operator algebras.  
Then there exists a unitary $u \in \Delta(B^{**})$ with $u^*$ a quasimultiplier of $B$,
 such that if $B(u)$ is 
the Jordan operator algebra in Lemma {\rm \ref{qm}}, then $T$ considered as a map into
$B(u)$ is an isometric surjective Jordan   homomorphism.
 \end{proposition}

  \begin{proof}  The proof of this is easy from  \cite[Corollary 2.8]{AS}:
$u = T^{**}(1)$ is a unitary
in $\Delta(B^{**})$ and since $T^{**}$ preserves the `partial triple product'
from \cite{AS} we have $T(a^2) = T^{**}(a 1 a) = T(a) u^* T(a) \in B$ for $a \in A$.
So $u^*$ is a quasimultiplier of $B$
and $T$ is a Jordan   homomorphism.  \end{proof}

{\bf Remarks.}  1) \ The last result may also be stated in terms of expressing $T$ in the form $T = u \theta(\cdot)$,
where $\theta : A \to C$ is an isometric surjective Jordan  algebra homomorphism
onto the Jordan subalgebra $C = u^* B$ of $D^{**}$, for $D$ as in the proof.  However we will improve on this
in Theorem \ref{ubsj} below.    We also remark that one may use $u^* B$ in place of $B u^*$ in the proof:
left  multiplication $L_{u^*}$ by $u^*$ is an isometric 
 Jordan   homomorphism from $B(u)$ onto $u^*B$.

\smallskip

2) \  With $B(u)$ as  $B$  with the new product above, as in the last proof, we have that  $B(u)$ is an approximately unital 
 Jordan operator algebra.    Examples like those in \cite[Example 6.6]{Brown2} show that one may not hope that the 
quasimultiplier  $u$ be a multiplier in $A$, even if $u = u^*$.

\smallskip

3) \  One sees Jordan algebra products given by quasimultipliers in  Proposition \ref{bsjoa} 
and e.g.\ in the proof of Theorem \ref{apch}.   Indeed we know from \cite[Section 2]{BNj} that all Jordan operator algebra 
products on an operator space $X$ are given by quasimultipliers: elements of the set  $$QM(X) = \{ u \in I(X) : x u^* x \in X \; \text{for all} \, x \in X \} .$$
This is related to ideas in \cite{KP} and \cite[Remark 2 on p.\  194]{BRS}.  
Here $I(X)$ is   the injective envelope of $X$ (see \cite[Chapter 4]{BLM} or \cite{Pau}).    One may ask  if the element $u$ in the bidual in
Proposition \ref{bsjoa} can 
be associated with an element of the injective envelope $I(X)$
by an explicit procedure, so that the expression $b u^* b$ in 
Proposition \ref{bsjoa} may be 
computed in the natural triple product of $I(X)$ (see e.g.\ 
the first paragraph of 4.4.7 in
\cite{BLM}).  
Related to this is the following question of the second author and Russo \cite{NRh}: Is the completely symmetric part of an operator space $X$ (see
\cite{NRh}) equal to $QM(X) \cap X$?    And if so, then is the restriction of the triple product on $I(X)$ to $X \times X_{cs} \times X$  equal to the partial triple product on $X$? 
One direction should be `easy': if $y \in QM(X) \cap X$ then
we get a Jordan algebra product on $X$.

\bigskip

Let $B$ be a $C^*$-algebra.   A quasimultiplier $w$ of $B$ in 
the sense of the present paper  is  the same as a quasimultiplier in the $C^*$-algebraic sense 
of \cite{Brown2} say, which requires that $b_1 w b_2 \in B$ for all $b_1, b_2 \in B$.
This follows by applying \cite[Theorem 4.1]{AP} to the real and imaginary part of $w$.  Checking item (ii) in that theorem, for example,
we get $a (w + w^* ) a \in B_{\rm s.a.}$ for $a \in B_{\rm s.a.}$.  Hence 
$a (w + w^* ) b \in B$ for all $a, b \in B$ by the cited theorem.   Similarly $\frac{1}{2 i} \, a (w - w^*) b \in B$, so that
$a w b \in B$ for all $a, b \in B$.  
\begin{proposition} \label{chopentr}  Suppose that $B$ is a $C^*$-algebra and that $u$ is a unitary in $B^{**}$ such that
$bu b \in B$ for all $b \in B$.  
Then $u \in M(B)$.   Also,  $b_1 u^* b_2$ is a $C^*$-algebra product on $B$, with corresponding involution $u b^* u$.  
That is, $B(u)$ is a $C^*$-algebra.  \end{proposition}   

\begin{proof}   The first assertion follows from the remarks above and by   \cite[Proposition 4.4]{AP} (see also \cite[Corollary 5.10.107]{Rod2}).
The remaining assertions are then straightforward.  (We thank the referee for these references; we had included another proof in a previous draft).
\end{proof}

{\bf Remark.}  The last result fails if we replace `unitary' by `tripotent' (that is, a partial isometry).   There are examples 
in \cite{Brown} of quasimultipliers $u$ of a $C^*$-algebra which are selfadjoint partial isometries, but which are not multipliers, and $B(u)$ 
is an operator algebra but is not a $C^*$-algebra.  

\begin{corollary} \label{chop2} 
Suppose that $B$ is an approximately unital
Jordan operator algebra, and that $D$ is a $C^*$-algebra generated by $B$, and $A$ is the operator algebra in $D$ generated by $B$.
Let $u$ be  a unitary in $\Delta(B^{**})$  with $u^*$ a quasimultiplier of $B$ such that $B(u)$  
 is an approximately unital Jordan algebra.     Then $u$ and $u^*$ belong to the Jordan multiplier 
algebra $JM(B)$, and also to $M(D)$ and $M(A)$. \end{corollary} 
  
\begin{proof}
We may suppose that $B^{**}$  is a unital  Jordan 
subalgebra of $D^{**}$, which in turn is a von Neumann algebra on $H$.  Then  $B^{**} u^*$
is a unital
Jordan subalgebra of $B(H)$.   For $d \in D$ by
\cite[Corollary 2.18]{BWj}  there exists $b_0 \in D$ and 
$a \in B$ with $d = a b_0 a$, and we can choose $b_0 \in A$ if $b \in A$.    Then $d u^* d \in a b_0 B b_0 a \in D$,
and $d u^* d \in A$ if $d \in A$.   So $D u^*$ is a  Jordan operator algebra, and so is $Au^*$.  By \cite[Lemma 2.6]{BWj}  one may choose a partial cai $(e_t)$ for $Bu^*$,
then $e_t a u^* \to a u^*$ and $a u^* e_t  \to a u^*$ in norm
in $B(H)$ for $a \in B$.  For $d = a b_0 a \in D$ as above 
we see that $e_t d u^* = e_t a u^* u b_0 a u^* \to d u^*$ and
$d u^* e_t = a b_0 a u^* e_t \to d u^*$ in norm.
So $D u^*$ is an approximately unital
Jordan subalgebra.  Hence $D(u)$
 is an approximately unital Jordan operator algebra.   Thus   
$u \in M(D)$ by Proposition \ref{chopentr}.  For $x \in A$ we have $x u \in A^{\perp \perp} \cap D = A$.
Similarly, $x u^*, ux, u^*x \in A$, so that $u \in \Delta(M(A))$.     Similarly, for $x \in B$ we have
$u \circ x \in B^{**} \cap D = B$, so that $u \in JM(B)$.  Similarly  we have $u^* \in  JM(B)$.
\end{proof}

\begin{theorem} \label{ubsj}    Suppose that $T : A \to B$ is an isometric surjection between approximately unital Jordan operator algebras.
Suppose that $B$ is  (isometrically) a Jordan subalgebra of an (associative) operator algebra (resp.\ $C^*$-algebra) $D$,
and that $B$ generates $D$ as an operator algebra (resp.\  $C^*$-algebra).
Then there exists a unitary $u \in \Delta(JM(B))$ which is also in
$\Delta(M(D))$,   
and  there exists an
isometric surjective Jordan  algebra homomorphism
$\pi : A \to C = B u^*$, such that $$T = \pi(\cdot) u .$$
Since $u^*$ is a quasimultiplier of $B$ in $M(D)$,  $C = B u^*$ is a Jordan subalgebra of $D$.
In addition, 
 \begin{itemize} \item  [(1)]  If $B$ is an  (associative) operator algebra then we may take $C = B$  above
(which also equals $D$ in the `non-respectively' case).
\item  [(2)]  If $T$ is a complete isometry and $B$ is completely isometrically a Jordan subalgebra of $D$
then $\pi$ is a complete isometry.
\item [(3)]  $C \subset \overline{{\rm Span}}(\{ b_1 b_2 : b_1, b_2 \in B \})$.  
 \end{itemize} 
\end{theorem}

  \begin{proof}  As in the  proof of Proposition \ref{bsjoa},
  $T^{**}(1) = u$ is a  unitary in $\Delta(B^{**})$,   with $u^*$ a quasimultiplier of $B$.  Hence as in that 
proof and the first remark after it, $C = B u^*$ is
a Jordan subalgebra of $D^{**}$, and $\pi = T(\cdot) u^*$ is
 an isometric surjective Jordan  algebra homomorphism from $A$ onto 
$B u^*$.    By Corollary \ref{chop2} it follows that
$u$ is in $\Delta(JM(B))$ and in $\Delta(M(D))$ so that $B u^*$ is
a Jordan subalgebra of $D$.  
For (1), we have $B = B u u^* = B u^*$ since $u \in \Delta(M(B))$.  Item (2) is obvious.
For (3) let $F =  \overline{{\rm Span}}(\{ b_1 b_2 : b_1, b_2 \in B \})$.  Because there is a net
$(b_t)$ in $B$ with weak* limit $u^*$, we have  
$C \subset   F^{\perp \perp} \cap D = F$.  (Similarly  $C \subset \overline{{\rm Span}}(\{ b_1 b_2^* : b_1, b_2 \in B \})$.)   \end{proof} 

{\bf Remarks.}  1)\ Of course as is usual with noncommutative Banach-Stone theorems, we can also write the unitary $u$ on the left:
$T = u \, \theta(\cdot)$.    To see this simply set $\theta = u^* \pi(\cdot) u$.

\smallskip

2)\ There is  a slight error in the converse direction 
of the proof of the Banach-Stone type result Proposition 6.6 in 
\cite{BNp}, in the Jordan algebra case.  To fix this, appeal to \cite[Corollary 2.8]{AS}, after noting  that $T(1) = u \geq 0$ by e.g.\ our Lemma \ref{ispos}, so that $u = 1$.

\section{Projections on Jordan operator algebras} \label{cp}

In this section we  give variants of
the
results in Section 2 in our paper \cite{BNp}.    However our maps are no longer completely contractive, and our spaces 
are usually  Jordan operator algebras.
The following,  which is due to Effros and St{\o}rmer \cite{ES} in the case that $P(1) = 1$,  shows what happens in the selfadjoint Jordan  case. 

\begin{theorem} \label{frs}  If $P : A \to A$ is a  positive 
contractive 
projection on a $JC^*$-algebra $A$ then $P(A)$ is a $JC^*$-algebra in the $P$-product,
$P$ is still positive as a map into the latter $JC^*$-algebra,
and $P(P(a) \circ P(b)) = P(a \circ P(b))$ for all $a, b \in A$. 
If in addition $P(A)$ is a Jordan subalgebra of $A$, then $P$ is a Jordan conditional expectation:
$P(a \circ P(b)) = P(a) \circ P(b)$ for $a, b \in A$. 
\end{theorem}

\begin{proof}  If $A$ is unital and $P(1) = 1$ then this is well known, following  from the $JC$-algebra case in \cite[Lemma 1.1 and Theorem 1.4]{ES}.
The general case  follows from the unital case by appealing to Theorem \ref{cepro}.   However 
we will give a second proof.  We do not claim that this proof is better, we simply offer it as an alternative that
may contain some useful techniques.  In addition we will need later the fact in the Claim below.   Again the claim 
 follows from the unital case in \cite{ES} and Theorem \ref{cepro}, but a different  proof may be of interest.

By considering $P^{**}$ we may assume that $A$ is unital. 
We have that  $P(A)$ is a $J^*$-algebra in a new triple product $P(\{ Px, Py, Pz \})$ by \cite[Theorem 2]{CPP}, and by \cite[Corollary 1]{FRAD} we have that \begin{eqnarray} \label{CE}
P(\{ Px, y, Pz \}) =  P(\{ Px, Py, Pz \}) = P(\{ Px, Py, z \}) 
\end{eqnarray}
for all $x,y,z \in A$.  This is effectively Kaup's contractive projection theorem \cite{Kaup} (see also \cite[Theorem 5.6.59]{Rod2}).
By Lemma \ref{ispos} we have $0 \leq P(1) \leq 1$.
Claim: $P(P(1)^n) = P(1)$  for all  $n \in \Ndb$.   By the above,
$P(P(1)^3) = P(1 P(1)^2) = P(P(1)^2)$.
Assume for induction that $P(P(1)^{n-1} - P(1)^{n}) = 0$ for some $n \geq 2$.
Then $$P(P(1)^{n} - P(1)^{n+1}) \leq P(P(1)^{n-1} - P(1)^{n}) = 0.$$ 
It follows that
\begin{eqnarray} \label{powers}
P(P(1)^n) = P(P(1)^2) , \qquad n \geq 2 .
\end{eqnarray}

Next suppose that $(p_n(t))$ is a sequence of polynomials with no constant term converging uniformly to  $t^{\frac{1}{3}}$ on $[0,1]$.
Then $p_n(t^3) \to t$.  Using (\ref{powers}) above and the contractivity of $P$, by the spectral theorem,  
\[
\| P(p_n(P(1)^3) - P(1)) \| \leq \| p_n(t^3) - t \|_\infty \to 0. 
\]
Again by (\ref{powers}), we have $P(p_n(P(1)^3) )= P(p_n(1) \, P(1)^2) \to  
P(P(1)^2)$. Thus $P(P(1)^2) = P(1)$. 
We have now proved the Claim.

It follows that $u = P(1)$ is a tripotent  (that is, a partial isometry) 
in the new triple product.  Let $x \in A_{+}$.  By (\ref{CE}) above we have $$P(\{1-u,P(x), y\}) = 0, \qquad  y \in P(A) .$$
Hence 
\begin{eqnarray*}
P(x) &=&P(uP(x)u) + 2P(\{1-u, P(x), u \}) + P((1-u)P(x)(1-u)) \\
&=& P(uP(x)u) + P((1-u)P(x)(1-u)). 
\end{eqnarray*}
However, by the Claim, $P((1-u)P(x)(1-u)) \le \| x \| \, P((1-u)^2) = 0$. Since this is true for all
$x \in A_{+}$, 
\[
P(x) =P(uP(x)u) = P(u P(uP(x)u) u) = P(\{u,P(\{u,P(x),u\},u\})
\]
for all $x \in A$. Hence $u$ is a unitary tripotent in the new triple product on $P(A)$.  Indeed $uu^*xu^*u = x$ implies that 
$$2\{u,u,x\} = uu^*x + xu^*u = uu^*uu^*xu^*u + uu^*xu^*uu^*u = 2uu^*xu^*u=2x.$$

If a $J^*$-algebra $Z$ has a 
unitary tripotent  $u$ then it is a $JC^*$-algebra with product $\{a,u,a\}$ and involution
$\{u,a,u\}$, for all $a \in P(A)$.   In our case the latter product is 
$P(au^*a) = P(a^2)$.   Also, $\{u,a,u\}= P(ua^*u)=P(a^*) = P(a)^*$, using the fact that 
 $P$ is selfadjoint. So $P(A)$ is a $JC^*$-algebra with the $P$-product $P(a^2)$ and with the old involution
$P(a)^*$.  Hence $P$ is still selfadjoint.   Also,
$P$ is still positive since it is unital and contractive. 
It also clear from (\ref{CE}) that for $a, b \in A_{\rm sa}$ we have
$$P(P(a) \circ P(b)) = P(\{ u , P(a), P(b) \}) = 
 P(\{ u , P(a), b \}) = \frac{1}{2}P(u P(a) b + b P(a) u).$$
 Similarly, $$P(P(a) \circ P(b)) = P(\{ u ,  P(b), P(a) \}) = P(\{ u ,  b, P(a) \}) = \frac{1}{2} P(u b P(a)  + P(a) b u).$$
 Taking the average, $$P(P(a) \circ P(b)) = \frac{1}{2}(P(u \circ (bP(a))) + P(u \circ (P(a) b)) )
 = P(u \circ (P(a) \circ b)).$$
  By Lemma \ref{ishe2} we have 
$P(u \circ (P(a) \circ b)) = P(P(a) \circ b)$, so that $P(P(a) \circ P(b)) =  P(P(a) \circ b)$ as desired.
 \end{proof} 
 
 We thank the referee for identifying some mistakes and omissions in the last proof, which we have repaired.
 
 As we said in Section 2 of \cite{BNp}, 
projections on operator algebras with no kind of approximate identity
can be very badly behaved, thus we say little about 
such algebras.  However one can pick out a `good part' of such a projection:

\begin{proposition} \label{reau}  Let $P : A \to A$  be a real
 positive contractive map (resp.\  projection)
on a Jordan operator algebra $A$ (possibly with no kind of approximate identity).
There exists a largest approximately unital Jordan subalgebra $D$ of $A$,
and it is a   hereditary subalgebra  of $A$.
Moreover, $P(D) \subset D$, and the restriction $P'$ of $P$ to $D$ is a 
 real  positive  contractive map (resp.\  projection)
on $D$.   In addition,  $P'$ is bicontractive (resp.\  symmetric) if 
$P$ has the same property.  \end{proposition}  

\begin{proof}     This follows as in \cite[Proposition 2.1]{BNp}, but using \cite[Corollary 4.2]{BWj}
in place of \cite[Corollary 2.2]{BRord}.
 \end{proof}

\begin{corollary} \label{ipow}  Let $P : A \to A$ be a real positive contractive
projection on a unital Jordan operator algebra.  Then $P(P(1)^n)  = P(1)$ for $n \in \Ndb$.  In addition, $P(1)$ is a projection
in $A$ if and only if $P(1)^2 \in P(A)$.
\end{corollary}

\begin{proof}     Note that  $P$ restricts to a positive contractive projection
from $\Delta(A)$ to $\Delta(A)$, and $P(1) \geq 0$ by Lemma \ref{ispos}. 
So we may assume that $A$ is a unital $JC^*$-algebra, and then
$P(P(1)^n)  = P(1)$ was established in the proof we gave of Theorem \ref{frs} (or it may be deduced from
that result).  
 If $P(1)^2 \in P(A)$ then we deduce that $P(1)^2 = P(1)$, so that  $P(1)$ is a projection. 
 \end{proof}  

In the sequel we will often restrict to the case that $P(1)$ or $P^{**}(1)$ is a projection.   This is automatic for example  if 
$P$ is a real positive bicontractive or symmetric 
projection on an approximately  unital Jordan operator algebra (by the proof of \cite[Lemma 3.6]{BNp}).  

\begin{lemma} \label{keride} 
Let $A$ be an
approximately unital  Jordan operator algebra, and suppose that  $P:A \rightarrow A$ is a  projection on $A$.
Then  ${\rm Ker}(P)$  is a Jordan ideal of $A$   if and only if  $P(a^2) = P(P(a)^2)$ for $a \in A$, that is  if and only if 
 $P$ is a Jordan homomorphism with the $P$-product on $P(A)$. 
In this case $P(A)$ with the $P$-product is Jordan isomorphic to $A/{\rm Ker}(P)$, and this isomorphism is isometric if 
$P$ is a contraction.  \end{lemma} 

 \begin{proof} 
For $a, b \in A, P((a-P(a)) \circ b) = 0$  if and only if  $P(a \circ b) = P(P(a) \circ b)$.   So Ker$(P)$  is a Jordan ideal of $A$   if and only if  $P(a \circ b) = P(P(a) \circ b)$ for all $a, b \in A$, which holds  if and only if $P(a \circ b) = P(P(a) \circ P(b))$  for all $a, b \in A$.
That is,   if and only if  $P$ is a Jordan homomorphism with the $P$-product on $P(A)$.   We leave the rest as an exercise.
\end{proof}

\begin{lemma} \label{zelo}  Let $A$ be a unital Jordan  operator algebra,
and $P : A \to A$ a contractive projection, such that ${\rm Ran}(P)$ contains an orthogonal  
projection $q$ with $P(A)  \subset q P(A) q$.   Then $q = P(1_A)$.
\end{lemma}

\begin{proof}  This follows as in \cite[Proposition 2.6]{BNp}, but using that the identity is an extreme point of the 
unit ball of any unital Jordan operator  algebra (since it 
is so in the generated $C^*$-algebra). 
\end{proof}

The following is a converse to the previous result, uses a similar proof, and gives a little more:

\begin{lemma} \label{cut2unit1}  Let $A$ be  a  unital Jordan operator algebra,
and let  $P : A \to A$ be a contractive real positive  projection. If $P(1) = q$ is a projection  then $P(A) = q P(A) q$.   In particular,
$q \circ P(A) = P(A)$, indeed $q$ is the identity for the unital operator space $q P(A) q$.    
We also have $$P(x) = P(qxq) = q P(qxq) q = qP(x)q$$ for $x \in A$. 
Thus $P$ restricts to a unital contractive  (real positive) projection on $qAq$, and $P$ is zero on the `rest' of $A$,
that is, on $q^\perp A q^\perp + \{ q^\perp a  q+ qa q^\perp : a \in A \}$.  \end{lemma}

\begin{proof}  This follows from  Lemma \ref{ishe3} with $e = 1$, and then 
with $e = q$.  A second proof: 
We have  $P({\mathfrak F}_A) \subset {\mathfrak F}_A$ by Theorem \ref{ocp}.  
Fix $x \in {\rm Ball}(A)$. 
 Then $P(1 \pm x) = q \pm P(x)$, and
$\| 1- q \pm P(x) \| \leq 1$.   This 
forces $\| 1- q \pm q^\perp P(x) q^\perp  \| \leq 1$, and 
since $I_K$ is an extreme point of $B(K)$ we see that 
$q^\perp P(x) q^\perp = 0$.   Looking at the matrix of $P(x)$ with respect to $q^\perp$
and using  $\| 1- q + q^\perp P(x) q^\perp  \| \leq 1$ we also see that $P(x) = qP(x) q$.
Now appeal to Lemma \ref{ishe2}.  
\end{proof} 

The last lemma may be viewed as a `reduction' to the case that $P$ is unital.   In the approximately unital case things are more 
difficult, since $qAq$ may not be a subset, let alone a subalgebra, of $A$.   Here $q = P^{**}(1) \in A^{**}$.
One might hope that $q$ might be some kind
of multiplier of $A$, but this is not usually the case.  The solution to this difficulty is found in the notion of  hereditary subalgebra,
and the notion of `zeroing extension' of maps on HSAs which we saw 
in and above Theorem \ref{hsaex}, and whose application to projections we shall describe after Lemma \ref{pretr2}.

\begin{lemma} \label{pretr2}   Let $A$ be an approximately 
unital Jordan operator algebra,
and $P : A \to A$ a contractive projection.  Let $q = P^{**}(1)$.   The following are equivalent:
\begin{itemize}  
\item [(i)]  $q$ is a projection
and $P$ is real positive.
\item [(ii)]  $q^2 \in {\rm Ran}(P^{**})$ 
and $P$ is real positive. 
\item [(iii)]  ${\rm Ran}(P^{**})$ contains an orthogonal 
projection $r$ such that $P(A) \subset r P(A) r$.
\end{itemize}
If these hold then $r = q = P^{**}(1)$, and the bidual of  ${\rm Ran}(P)$ is a unital operator space with  identity $q$.
Also,   $q$ is an open projection 
for $A^{**}$ in the sense of our introduction, so that $D = \{ a \in A : a = qaq \}$ is a 
hereditary subalgebra, and  $D^{**} = q A^{**} q$.  \end{lemma}

\begin{proof}    The equivalence of (i) and (ii) follows from Corollary \ref{ipow}.  
 Let $Q = P^{**}$,  a  contractive projection  on $A^{**}$.  We can replace $Q$ by $P$ below
if $A$ is unital.   
If $P(A) \subset q P(A) q$ for a projection $q \in Q(A^{**})  q$ then $Q(A^{**}) = \overline{P(A)}^{w*} =  q Q(A^{**})  q$ by standard weak* approximation
arguments.  By Lemma \ref{zelo} we have $Q(1) = q$.  Since
Ran$(Q)$ is a unital operator space (in $qA^{**}q$) with identity $q$, and $Q(1) = q$, we see by the line before Corollary \ref{trivj} that
$Q$, and hence also $P$,  is real  positive as a map into $qA^{**}q$.  

Conversely, suppose that $P$ is real positive
and $Q(1) = q$ is a projection.  Then 
$Q(q^\perp) = 0$.   
  Then  Lemma \ref{cut2unit1}  gives
$P(A) = q P(A) q$.   

If  $(e_t)$ is a cai for $A$ then $e_t \to 1$ weak*, so that 
$P(e_t) = q P(e_t) q \to q P^{**}(1) q = q$ weak*.   Thus $q$ is an open projection in $A^{**}$
in the sense of  \cite{BNj}, and $D$ is a 
hereditary subalgebra as stated.   The rest is clear.  
\end{proof} 

{\bf Remark.}  We remark that there is a mistake in the
statement of the matching result in \cite{BNp}, namely Proposition 2.7 there.  After the phrase
`$P^{**}(1)$ is a projection' there the condition `and $P$ is real completely positive' should be added, similarly to the statement of (ii) 
above.  This mistake led to an error in \cite[Theorem 3.7]{BNp}:
in the statement of that result the phrase `real completely positive' should be deleted in the last line (see our 
Theorem \ref{trivch2c} for a way to state 
 the converse direction of that theorem).   
 
\bigskip

We now describe the   `zeroing extension' of maps on HSAs which we saw 
in and above Theorem \ref{hsaex},  as applied to projections.     
Suppose that 
 $P : A \to A$ is a real positive 
contractive projection on  an approximately 
unital Jordan operator algebra, such that 
$P^{**}(1) = q$ is a projection.   By Lemma \ref{pretr2},  $D = \{ a \in A : a = qaq \}$ is a 
hereditary subalgebra of $A$.  Suppose that
 $D$ is represented nondegenerately
on $B(K)$ for a Hilbert space $K$ (see \cite[Section 6]{BWj}).   Then the 
restriction $E$ of $P$ to $D$, viewed as a map 
$D \to B(K)$ satisfies the cai condition in Theorem \ref{hsaex}, since if $(e_t)$ is
a partial cai for $D$ then $e_t \to q$ weak* in $D^{**}$ and $e_t \to I_K$ weak* in $B(K)$.
That is, $q$ acts as the identity on $K$.   Hence  $P(e_t) \to I_K$ weak* since 
$P^{**}(q) = q$.
Thus by that theorem $E$ extends uniquely to 
a contraction from $A$ to $B(K)$ satisfying a similar cai condition spelled out in Theorem \ref{hsaex}.  
This is the `zeroing extension', which
kills `the complement' 
$\{ q^\perp a  q^\perp + q^\perp a  q+ qa q^\perp : a \in A \}$ 
of $D$.   However by uniqueness
this extension must be $P$, since $P^{**}(1) = q = 1_{D^{**}}$, which acts as the identity on $K$.
Thus $P$ is the zeroing extension of the  
restriction of $P$ to   $D$.   One may view this as a 
  `cut down to the unital case' procedure:  $P$ is `unital' on the HSA.  That is, $E^{**}$ is a unital projection on $D^{**}$, and 
$P^{**}$ and $P$ are zero on the `complement' of the HSA since $P(a) = P(q aq)$ for all $a \in A$.  

The following theorem  is the contractive projection version of \cite[Proposition 2.7]{BNp}.  

\begin{theorem} \label{tr2}  Let $A$ be an approximately 
unital Jordan operator algebra,
and $P : A \to A$ a real positive contractive projection.   Suppose that  $q = P^{**}(1)$ is a projection.  
We have  
$$P(a) = q P(a)q  = P^{**}(qaq), \qquad a \in A,$$ (and we can replace $P^{**}$ by $P$ here if $A$
is unital).   Hence $P(A) = q P(A) q = P^{**}(qAq)$.   Also, 
\begin{itemize}  
\item [(1)] $P$ `splits' as the sum of  the zero map on $q^\perp A q^\perp + \{ q^\perp a  q+ qa q^\perp : a \in A \}$, and
a real positive contractive projection $P'$ on $qAq$ with range equal to 
$P(A)$.   This  projection $P'$ on $qAq$ is unital if $A$ is unital.   
\item [(2)]   $P$ restricts to a real positive  contractive  
projection $E$ on the hereditary subalgebra
$D$ supported by $q$ (see Lemma {\rm \ref{pretr2}}).  We have $E(D) = P(A) \subset D$, and $E^{**}$ is unital:
$E^{**}(q) = q$.  
\item [(3)]    $P$ is the zeroing extension of $E$.  
 \end{itemize}
\end{theorem}

\begin{proof}    The displayed formula follows from the last result and  Lemma \ref{cut2unit1}
 applied to $P^{**}$.  
Item (3) is discussed above the theorem.   
 We leave the rest as an exercise.  
\end{proof}

\begin{lemma} \label{Araz}  Let $A$ be  an approximately  unital Jordan 
operator algebra,
and let  $P : A \to A$ be a contractive real positive projection.   Then $$P(a \circ b) = P(a \circ P(b)) , \qquad a \in P(A), b \in \Delta(A).$$  In particular, if $A$ is unital 
and  $q = P(1)$ then $a = P(a \circ q)$ for $a \in P(A)$.
   \end{lemma}

\begin{proof}    By considering $P^{**}$ we may suppose that $A$ is unital.  By \cite[Theorems 2.6 and 2.7]{AS} the symmetric part of $A$ is $\Delta(A) = A \cap A^*$, and the partial triple product
on $A$ is $\{ a , b , c \} = (a b^* c + c b^* a)/2$ for $a,  c \in A, b \in \Delta(A)$.   The restriction of $P$ to $\Delta(A)$ is real positive, hence is a positive map into 
$\Delta(A)$ as in Lemma \ref{ispos}.  
By \cite[Proposition 5.6.39 (i) and (ii)]{Rod2}, $P(\{ a , b , c \}) = P(\{ a , Pb , c \})$ for $a,  c \in P(A), b \in \Delta(A)$.   (We have
used here the fact from Lemma \ref{ispos} that $P(\Delta(A)) \subset \Delta(A)$.)
 Setting $c = 1$ we have
$$P(a \circ b) = P(a \circ P(b^*)^*) = P(a \circ P(b)),  \qquad a \in P(A), b \in \Delta(A),$$  since $P$ is positive on $\Delta(A)$.  Setting $b = 1$ gives 
$P(a) = a = P(a \circ q)$ for $a \in P(A)$.   
  \end{proof}

{\bf Remarks.}  1) \ We thank J. Arazy for the main insights in the previous proof.

\smallskip

2) \  As we said in the introduction, a couple of paragraphs after Theorem \ref{apch}, 
$P(ab) \neq P(a) b$ in general for a unital (even commutative) operator algebra $A$ and contractive unital (hence real
positive) projection from $A$ onto a subalgebra containing $1_A$, and $a \in A, b \in P(A)$. 
Thus one cannot hope for Jordan operator algebra variants of the centered equation in \cite[Theorem 2.5]{BNp}
or the conditional expectation assertion in \cite[Corollary 2.9]{BNp}.   
Also Ker$(P)$ is not a Jordan ideal in this case.   However with respect to these results in the case of contractive morphisms we will see later that
things become  much better under extra hypotheses (such as those in Theorem \ref{tr}, Corollary
\ref{ipscor}, or
Theorem \ref{trivch2c}). 

\smallskip

3) \ Lemma \ref{Araz} may be used to give a different proof of a variant of Theorem \ref{tr2}.

\bigskip 

The following is a variant  of the Choi-Effros result 
referred to in the introduction.     The $JC$-algebra version of the result is due to Effros and St{\o}rmer \cite{ES}.  

\begin{theorem} \label{tr}  Let $A$ be an 
approximately unital Jordan operator algebra,
and $P : A \to A$ a contractive real positive projection.
Suppose that ${\rm Ker}(P)$
 (resp.\  ${\rm Ker}(P) \cap \, {\rm joa}(P(A))$) 
is
densely spanned by the real positive elements which it contains.
Then 
the range $B = P(A)$ is an
approximately unital Jordan operator algebra with product $P(x \circ y)$.    If $A$ is unital then $P(1)$ is the
 identity for  the latter product.
Also $P$ is a (real positive) Jordan homomorphism with respect to this
product: $$P(a \circ b) = P(a \circ P(b)) = P(P(a)  \circ P(b))$$ for $a, b$ in $A$
 (resp.\ in ${\rm joa}(P(A))$).  
\end{theorem}

\begin{proof} We may assume (using Lemma \ref{paauj}) that $A = {\rm joa}(P(A))$.  By Lemma \ref{genf} it follows that  
${\rm Ker}(P)$ is an  approximately unital Jordan ideal
in $A$.  
By Lemma \ref{keride} we deduce that $P(A)$ is an  approximately unital  Jordan operator algebra
and $P$ is a  Jordan homomorphism, both  with respect to the new 
product.   That $P$ is real positive  with respect to the new 
product follows e.g.\ since $P^{**}$ is a unital contraction into
$qA^{**} q$ and hence is real positive. 
\end{proof} 

{\bf Remarks.}  1) \ One may `weaken' the condition  in Theorem  \ref{tr} about being `densely spanned by the real positive elements which it contains',
to being `contained in the closed Jordan algebra generated by the real positive elements it contains', or even 
 being `contained in the hereditary subalgebra generated by the real positive elements it contains'.    The proof in these latter cases is only 
slightly more difficult, one needs to appeal to Lemma \ref{genf}  (2). 
We used quotes around `weaken' because once we know (by that lemma) that Ker$(P)$ is approximately unital then
it follows  (e.g.\ as is clear from that lemma) 
that it is the span of  the real positive elements which it contains.   

\smallskip

2) \    
Note that if $P$ is a positive projection on a $JC^{\ast}$-algebra, then the
 condition requiring $D = {\rm Ker}(P) \cap \, {\rm joa}(P(A))$ be densely spanned by the real positive elements which it contains,  is always true.  Indeed, 
note that $D$ is selfadjoint.
By \cite[Lemma 1.2]{ES}
applied to $P^{**}$, if $x \in D_{sa}$ then $P(x^{2}) = P(P(x)^{2}) = 0$, and so $x^{2} \in D$. Hence $D_{sa}$ is a Jordan subalgebra of $A_{sa}$. Thus it is spanned by positive elements
by the usual functional calculus (in $C^*(x)$ for $x \in D_{sa}$). Hence $D = D_{sa} + i D_{sa}$ is spanned by its positive elements.  
Thus the `respectively' case of  Theorem \ref{tr} generalizes  Theorem \ref{frs} (and Theorem 1.4 of \cite{ES}), that the range of a positive projection on a $JC^{\ast}$-algebra is isometric to a $JC^{\ast}$-algebra.  

However, for a positive projection $P$ on a $C^{\ast}$-algebra, Ker$(P)$ need not have nonzero, nor be spanned by its, positive elements
(even if $P$ is a state).  So we feel the `respectively' case of Theorem \ref{tr}  is a suitable generalization of 
 Theorem \ref{frs} (and Theorem 1.4 of \cite{ES}).

\smallskip

3)\   Key to this proof is obtaining that  Ker$(P)$ is an ideal.  If $P$ is also completely contractive (and unital)  then we showed in \cite[Corollary 2.11]{BNp}  (see Corollary \ref{newaei}) that Ker$(P)$ is always an ideal in $A$, if $A$ is generated by $P(A)$.
We do not know if this is  correct with the word `completely' removed.  However if $A$ is
an approximately unital  associative operator algebra and $P$ is a real positive
`conditional expectation' then  Ker$(P)$ is an ideal in $A$  if $A$ is generated by $P(A)$, by Theorem \ref{AB} and the remark after it
(since $P(A)$, and hence  the operator algebra generated by $P(A)$, is contained in the algebra $B$ there).

\begin{corollary} \label{trps}  Let $A$ be an 
approximately unital Jordan operator algebra,
and $P : A \to A$ a contractive projection whose range is a Jordan subalgebra.
Suppose that ${\rm Ker}(P)$
 (resp.\  ${\rm Ker}(P) \cap \, {\rm joa}(P(A))$)
 is
densely spanned by the real positive elements which it contains.
Then $P$  is  real positive if and only if $P(A)$ is approximately unital, and
in this case $P$ is a Jordan conditional expectation:  $$P(a \circ P(b)) =P(a)  \circ P(b)$$ for $a, b$ in $A$ 
 (resp.\ in ${\rm joa}(P(A))$). 
\end{corollary}

\begin{proof}  
 If  $P(A)$ is approximately unital, then $P$  is  real positive  by Theorem \ref{pretr2}.
The rest is clear from Theorem \ref{tr}.
\end{proof} 

Following the last proofs, but using the operator algebra case of the results used,  yields: 

\begin{theorem} \label{troa}  Let $A$ be an 
approximately unital  operator algebra,
and $P : A \to A$ a contractive real positive projection.
Suppose that ${\rm Ker}(P)$
 (resp.\  ${\rm Ker}(P) \cap \, {\rm oa}(P(A))$) 
is
densely spanned by the real positive elements which it contains.
Then 
the range $B = P(A)$ is an
approximately unital operator algebra with product $P(x   y)$.    If $A$ is unital then $P(1)$ is the
 identity for  the latter product.
Also $P$ is a  homomorphism with respect to this
product: $$P(a b) = P(a  P(b)) = P(P(a) b) = P(P(a)  P(b))$$ for $a, b$ in $A$
 (resp.\ in ${\rm oa}(P(A))$).     Finally, if $P(A)$ is a subalgebra of $A$ then 
the last quantity in the last centered equation equals $P(a)  P(b)$.
\end{theorem}

{\bf Remark.}  As in the first Remark after Theorem \ref{tr}, one may `weaken' the condition  in Theorem  \ref{troa} about being `densely spanned by the real positive elements which it contains',
to being `contained in the closed algebra (or even HSA) generated by the real positive elements it contains'.  

\bigskip

For the readers convenience we mention the following `partially selfadjoint' results which are proved in \cite{Brp}:
 
\begin{theorem} \label{ipscor}  Let $A$ be an 
approximately unital Jordan operator algebra,
and $P : A \to A$ a  real positive contractive projection with $P(A) \subset \Delta(A)$.
Then $P(A)$ is a $JC^*$-algebra in the $P$-product, and the restriction of $P$ to ${\rm joa}(P(A))$
is a Jordan $*$-homomorphism onto this  $JC^*$-algebra.
In this case $P$ is a Jordan conditional expectation with 
respect to the $P$-product:  $$P(a \circ P(b)) =P(P(a)  \circ P(b))$$ for $a, b$ in $A$. 
 \end{theorem}

{\bf Remark.}    A main ingredient from   \cite{Brp} in the proof of the last result is the following 
lemma:  Let $A$ be  an approximately  unital  Jordan
operator algebra,
and let  $P : A \to A$ be a contractive real positive projection such 
 that $P(A)$ is a Jordan
operator algebra with $P$-product.   If 
$q \in P(A)$ is a projection in the $P$-product then
  $P(qaq) = P(q P(a) q)$ and $P(a \circ q) =
P(P(a) \circ q)$ for all $a \in A$.  (These assertions  follow from  Lemma \ref{ishe3}.)
If further $A$ is weak* closed and $P$ is weak* continuous then  $$P(a \circ b) =
P(P(a) \circ b), \qquad a \in A, b \in \Delta(P(A)).$$

\begin{corollary} \label{ipscor2}  Let $A$ be an 
approximately unital  (associative) operator algebra,
and $P : A \to A$ a  real positive contractive projection with $P(A) \subset \Delta(A)$.   Then 
the $P$-product on $P(A)$ is associative (which happens
for example if $P(A)$ is an (associative) subalgebra of $A$) if and only if 
$P$ is completely contractive.  In this case $P$ is real completely positive, 
$P(A)$ is a $C^*$-algebra in the new product, and $O$ is a conditional expectation in the latter product:  
 $$P(P(a)P(b)) = P(aP(b)) = P(P(a)b), \qquad a, b \in A.$$
 \end{corollary}

{\bf Remark.}
 A contractive (real positive) Jordan conditional expectation need not be  completely contractive.
Merely consider $P(x) = \frac{1}{2}(x + x^\intercal )$ on $M_2$.
The same example also shows the necessity of the 
condition in the last result  that 
the new product on $P(A)$ is associative (even in the case that
$A$ is a $C^*$-algebra). 

\bigskip

The following result, also   from \cite{Brp}, is inspired by the selfadjoint case due to Effros et al (see e.g.\ \cite[Lemma 1.4]{ES}).
If $A$ is a unital operator algebra, and $P$ is  a unital contractive or
completely contractive projection
on $A$, define
 $$N = \{ x \in A : P(xy) = P(yx) = 0 \; \text{for all} \; y \in A \}.$$
If $A$ or $P$ is not unital, but $P$ is also real positive,
then we may extend $P$ to a unital contractive projection
on $A^1$, where $A^1$ is a unitization with $A^1 \neq A$, and 
set  $N = \{ x \in A : P(xy) = P(yx) = 0 \; \text{for all} \; y \in A^1 \}$.
   Then $N$ is clearly a closed ideal in $A$, and is also
a subspace of Ker$(P)$.  Define
$$B = \{ x \in A : P(xy) = P(P(x) P(y)) \, {\rm and} \, P(yx)  = P(P(y) P(x)) \; \text{for all} \; y \in A \}.$$  Then $N \subset B$ since if $x \in N \subset {\rm Ker}(P)$
then $P(xy) = 0 = P(P(x) P(y))$ for all $y \in A$.
Note too that $1 \in B$ if $A$ is unital and $P(1) = 1$.

\begin{theorem} \label{AB} If $P$ is a  real positive contractive  projection on an approximately unital operator algebra $A$, and 
$N, B$ are defined as above, 
then  $P(A) \subset B$ if and only if
$$P(P(a) b) = P(P(a) P(b)) = P(a P(b))  \; \text{for all} \; a , b \in A.$$
That is,  if and only if $P$ is a conditional expectation onto $P(A)$
 with respect to
the $P$-product.  This is also equivalent to $B = P(A) + N$.
If these hold then  $P(A)$
with the $P$-product is isometrically isomorphic
 to an operator algebra, $B$  is a subalgebra of $A$ containing
$P(A)$, and $P$ is a homomorphism from $B$ onto  $P(A)$
with the $P$-product.
\end{theorem}

{\bf Remark.}  Note that  in the last result $N = {\rm Ker}(P)$ iff
${\rm Ker}(P)$ is an ideal.  The latter
holds (by the associative algebra variant of 
Lemma \ref{keride}) if and only if   $B = A$, and then all of the conclusions of the last theorem hold.  Also, if $P$ is real completely positive and completely contractive then $P(P(a) b) = P(P(a) P(b)) = P(a P(b))$
for $a, b \in A$  as is proved in \cite[Section 2]{BNp}, so that the conclusions of the last theorem hold. 

\bigskip

Finally, we comment on the operator space version of the theory in the present section.  All of the results in 
\cite[Section 2]{BNp} stated for completely contractive projections on  operator algebras are true for
 for completely contractive projections on  Jordan operator algebras, with essentially unchanged
proofs.   As we said in the introduction, the main extra thing one needs to know for some of these proofs to go through is the following:

\begin{lemma} \label{ienv}  Let $A$ be an  approximately unital Jordan operator algebra, and let $A^1$ be its
unitization.  Then
  the injective envelope  $I(A) = I(A^1)$, and this  may be taken to be a unital $C^*$-algebra
containing $A$  
completely isometrically as a 
Jordan subalgebra.    Moreover, any injective envelope of the unitization of $A$ is an  injective envelope of $A$.  \end{lemma}
\begin{proof}    This was stated in Corollary 2.22 of \cite{BWj}.   Since the proof was rather terse we give more details.  
The idea is similar to the  proof of 
\cite[Corollary 4.2.8]{BLM}, but uses facts in and after \cite[Lemma 2.19]{BWj} about  SOT convergence of partial cai for a nondegenerate 
representation of an approximately unital 
Jordan operator algebra.  These give, in the notation used in  \cite[Corollary 4.2.8]{BLM},
that  the restriction of the functional $\varphi = \langle \Phi(\cdot) \zeta , \zeta \rangle$ on $B(H)$ to $A$,  is a state of $A$.  So the restriction of $\varphi$ 
to $A + \Cdb I_H$ has the same norm 1.   By   \cite[Lemma 2.20 (1)]{BWj}, $\varphi(I) = \langle \Phi(I) \zeta , \zeta \rangle
= 1$.   It follows that $\Phi(I) = I$, and the  rest is as in the proof  of \cite[Corollary 4.2.8]{BLM}.  For example,  
the operator system $R$ there, which is an injective envelope of $A$, is an injective envelope of $A^1$, and by the Choi-Effros 
result cited in our introduction, $R$ is a unital $C^*$-algebra.   \end{proof}

We obtain for example:

\begin{corollary} \label{newaei}   Let $P : A \to A$ be a unital completely contractive projection on a Jordan operator algebra.
If $P(A)$ generates $A$ as a  Jordan operator algebra, then ${\rm Ker}(P)$ is a Jordan ideal in $A$.  In any case,
if $D$ is the closed  Jordan algebra generated by $P(A)$ then ${\rm Ker}(P_{|D})$ is a Jordan ideal in $D$, and 
$P_{|D}$ is a Jordan homomorphism with respect to the $P$-product on its range. 
 \end{corollary}

\begin{proof} 
 We may assume that $P(A)$ generates $A$.
As in the proof of \cite[Corollary 2.11]{BNp} but using Lemma \ref{ienv}, we extend $P$ first to a unital completely contractive projection on a $C^*$-algebra $B$
($= I(A)$), then to a weak* continuous unital completely contractive projection  $\tilde{P}$   on a von
Neumann algebra $M$ ($=B^{**}$).    Let  $\tilde{P}$ also 
denote the restriction of the latter projection to the von Neumann algebra $N$ generated by $P(A)$ inside $M$.
  If $x \in (I-P)(A)$, then $\tilde{P}(x) = 0$, and so by \cite[Proposition 2.10]{BNp} we have
$ex = x e = 0$.  Thus $x \in e^{\perp} M e^{\perp}$ and $x \in e^{\perp} N e^{\perp}$.   So for  $y \in A$ we have 
$P(xy + yx) = P(exy e + eyxe) = 0$.   Hence  ${\rm Ker}(P)$ is a Jordan ideal.  
The  Jordan homomorphism assertion follows from Lemma \ref{keride}.   \end{proof}  

\begin{corollary} \label{sampj} Let $A$ be an 
approximately unital  Jordan operator algebra,
and $P : A \to A$ a completely contractive  projection which is also completely real positive. 
Then 
the range $B = P(A)$ is an
approximately unital Jordan operator algebra with product $P(x \circ  y)$.    If $A$ is unital then $P(1)$ is the
 identity for  the latter product.
Also $P$ is a Jordan conditional expectation with respect to this new 
product: $$P(a  \circ P(b)) = P(P(a) \circ  P(b)) , \qquad a, b \in A .$$     If $P(A)$ is a Jordan subalgebra of $A$ then 
the last quantity in the last centered equation equals $P(a)  \circ P(b)$.
\end{corollary}

\begin{proof}  (Sketch.) \ As usual we may assume $A$ is unital.
The first assertion follows from Corollary \ref{newaei} and the idea in the proof of Lemma \ref{keride} with $A$ 
replaced by joa$(P(A)$.   Alternatively, all the results here follow as in \cite[Section 2]{BNp} by extending $P$ 
to a  completely contractive completely positive projection $Q$ on $I(A)$.  One then appeals to the $C^*$-algebra
case for the matching results for $Q$, which yield our results when restricted to $A$.
 \end{proof}  

The following kind of complement exists in the mixed situation that $P$ is completely contractive but possibly not completely real positive:

\begin{theorem} \label{ccrp}  If $P : A \to A$ is a   real positive completely contractive 
projection on an approximately unital  operator algebra $A$ then with respect to the $P$-product
$P(A)$  is an approximately unital operator algebra. 
\end{theorem}

\begin{proof}    By standard arguments, $Q = P^{**}$ is a real positive completely contractive 
projection on the unital operator algebra  $M = A^{**}$.  
Following the proof of  \cite[Corollary 2.3]{BNj}, $Q(x y)$ defines an operator algebra product on $Q(M)$.
By Lemma \ref{Araz}, $q = Q(1)$ is a Jordan identity for
the new product. 
 Therefore $q$ is a projection and is
an identity of norm 1 for 
the new product by the discussion around formula (1.1) in  \cite{BWj}.
We note that the map $Q$ regarded as mapping into this operator algebra is a unital complete contraction, hence is
real completely positive.

Since $P(A)^{**} = P(A)^{\perp \perp} = P^{**}(A^{**})$  (see e.g.\ the proof 
of Lemma \ref{pretr2}), 
$P(A)$ is an approximately unital operator algebra with the $P$-product, by e.g.\ \cite[Proposition 2.5.8]{BLM}.  
 \end{proof}

{\bf Remark.}  One cannot hope for a complement to Theorem \ref{apch}
saying that the projections there are always conditional expectations
with respect to the new operator product on $P(A)$ coming
from Theorem \ref{apch}.  That is, one cannot hope for
$P(P(x) \circ y) = P(x \circ y)$
for all $x \in A$ and $y \in P(A)$ in (1) and (3) there,
or $P(P(x) y) = P(x y)$ 
 in (2).
It is easy to find counterexamples, for example the projection 
in the $2 \times 2$  matrix  example at the end of Section \ref{Symmpr}.  

\section{Symmetric projections}  \label{Symmpr} 
We first recall  the solution to the bicontractive and symmetric projection problem for $JC^*$-algebras, essentially 
due to deep work of Friedman and Russo, and St{\o}rmer \cite{FRAD, FRceJ, ST82}. Some of this hinges on Harris's Banach--Stone type theorem
for $J^*$-algebras \cite{Har2} mentioned in the introduction (where we mentioned the main facts about morphisms
between $JC^*$-algebras).  The following 
is essentially very well known (see the references above), but 
we do not know of a reference which has all of these assertions, or is in the formulation  we give:  

In the following result $M(A)$ is the multiplier algebra from \cite{Rod2}, 
that is the elements $x \in A^{**}$ with 
 $x \circ A \subset A$, or equivalently with $\{ x , y , z \} \in A$ for all $y, z \in A$.  These imply
that $y x^* y  \in A$ for all $y \in A$.  
If $A$ is a $C^*$-algebra then this is just the usual  $C^*$-algebraic multiplier algebra.  This follows from e.g.\ \cite[Proposition 5.10.96]{Rod2}
or a fact about $JM(A)$ from our introduction.   

\begin{theorem} \label{goq}  If $P : A \to A$ is a    
projection on a $JC^*$-algebra $A$ then 
$P$ is bicontractive  if and only if  $P$ is symmetric.  Moreover  $P$ is bicontractive and 
positive  if and only if  there exists a central projection $q \in M(A)$ (indeed $q \circ a = qaq \in A$ for all $a \in A$) such that
 $P = 0$ on $q^\perp A q^\perp$, and there exists a Jordan $*$-automorphism $\theta$ of $qAq$   of period 2 (i.e.\ $\theta \circ \theta = I$)
so that $P = \frac{1}{2}(I + \theta)$ on $qAq$.     Finally, $P(A)$ is a $JC^*$-subalgebra of $A$, and 
$P$ is a Jordan conditional expectation.
\end{theorem}

\begin{proof}   
Clearly symmetric projections are bicontractive.
Conversely, if $P$ is bicontractive then by \cite[Theorem 2]{FRAD}
 $\theta =2P - I_A$ is a linear surjective isometric Jordan triple product preserving selfmap of $A$ (that is, a 
$J^*$-algebra isomorphism) 
with $\theta \circ \theta = I_A$, and $P = \frac{1}{2}(I_A + \theta)$.  So $P$ is
 symmetric.  If also $P$ is positive then $P$, and hence also 
$\theta$, is $*$-linear.     Let $Q$ be the extension of $P$ to the second dual.
Let $u = \theta^{**}(1)$, which is a selfadjoint unitary (a symmetry) in the $C^*$-algebra $D$ generated by $A^{**}$.
Since $\theta^{**}$ is a $J^*$-algebra isomorphism we have 
$u \in M(A)$, thus $u \circ A \subset A$.   So $q = Q(1) = (1+u)/2$ is a projection  and $q \circ A \subset A$ (that is, $q \in M(A)$).
    We have $u \theta^{**}(x) u = u \theta^{**}(x^*)^* u = \theta^{**}( 1 (x^*)^* 1) =  \theta^{**}(x)$.
So $u$, and hence also $q$, is central with respect to the product in $D$.   Thus $q \eta q = q \circ \eta$ for $\eta \in A^{**}$.
It follows that $A = q Aq + q^\perp A q^\perp$.

Since $Q(q^\perp) = 0$, if $a \in {\rm Ball}(A)_+$ then
$$P(q^\perp a q^\perp) \leq Q(q^\perp) = 0,$$
and so $P = 0$ on $q^\perp A q^\perp$.  Also, since $\theta(q) = 2P(q) - q = q$ and $\theta$ is  $*$-linear and 
 Jordan triple product preserving,
it follows that $\theta(qaq) = q \theta(a) q$ and  $P(qaq) = \frac{1}{2}(qaq + \theta(qaq)) = q P(a) q.$
Thus,
$P(qAq) = qP(A)q \subset qAq$, and
 the restriction of $P$ to $qAq$ is a unital bicontractive positive projection on a unital  $JC^*$-algebra.  
Also since $\theta(qaq) = q \theta(a) q$ for $a \in A$, as we had above,
  $\theta(qAq) = q Aq$.  Hence $\theta' = \theta_{\vert qAq}$ is 
a  unital isometric Jordan $*$-automorphism of $qAq$, and $P = \frac{1}{2}(I + \theta')$ on $qAq$.  
The converse is easy; the centrality of $q$ allowing it to suffice 
to check the conditions on each of the two orthogonal parts of $A$.

Finally  $P(A) = P(qAq)$, which consists of the fixed points of $\theta'$ and hence is a $JC^*$-subalgebra.   
Moreover if $a \in A, b \in P(A)$ we have $$P(a \circ b) = P(qa \circ b) =  \frac{1}{2}(qa \circ b + \theta'(qa \circ b)) = P(a) \circ b$$
since  $\theta'$ is a Jordan
homomorphism, and $b$ is fixed by $\theta'$.
\end{proof}  

The following is the solution to the symmetric projection problem
in the category of  approximately unital Jordan operator algebras.

\begin{theorem} \label{trivch2c}  Let $A$ be an approximately unital  Jordan operator algebra,
and $P : A \to A$ a symmetric real positive projection.
Then the range of $P$  is an approximately unital 
Jordan subalgebra of $A$ and 
$P$ is a Jordan conditional expectation.   Moreover, $P^{**}(1) = q$ is an open projection 
(in the sense of our introduction) in the Jordan 
multiplier algebra $JM(A)$,
and all of the conclusions of  Lemma {\rm \ref{pretr2}} and Theorem   {\rm \ref{tr2}} hold.   

Set $D$ to be the  hereditary subalgebra of $A$ supported by $q$, indeed $D = q Aq$, which contains $P(A)$.
Let $W = q^\perp A q^\perp + \{ q^\perp a  q+ qa q^\perp : a \in A \}$, which is a complemented subspace of $A$, indeed 
$A = D \oplus W$.  
There exists a period $2$ surjective  isometric Jordan isomorphism $\pi : D \to D$,
which is the restriction of a  period $2$ isometric selfmap of $A$,  
such that $$P =  \frac{1}{2} (I + \pi) \; \; \; {\rm on} \; D,$$ 
and $P = 0$ on  the complement $W$ of $D$ in $A$ (thus  $P(a) = P(qaq)$ for 
all $a \in A$).
The range of $P$, which equals $P(D)$, is exactly the set of fixed points of $\pi$  in $D$.   

Conversely, if $q$ is a projection in $JM(A)$,  $\pi$ is a period $2$  isometric Jordan automorphism of $D = qAq$, and if
$P =  \frac{1}{2} (I + \pi)$ on $D$ and $P = 0$ on the complement $W$ above of $D$ in $A,$  
then $P$ is  a symmetric real positive projection on $A$. 
\end{theorem}

\begin{proof}    As before  $P^{**}(1) = q$ is a projection, and 
all the conclusions of  Lemma \ref{pretr2} and Theorem  \ref{tr2}  are true for us.  
We will silently be using facts from these results below. In particular $q$ is an open projection, so supports
an approximately unital HSA $D$ of $A$ with $D^{\perp \perp} = qA^{**}q$.   
We know that $P(A) \subset D$.   Then  
$\theta = 2 P - I$ is a period 2  linear  isometric surjection on $A$.  
If $u = \theta^{**}(1) = 2 P^{**}(1) - 1 = 
2 q - 1$, then $u$  is a selfadjoint unitary (a symmetry).  If $D$ is a $C^*$-algebra generated by $A$, 
then $u$ and $q$ are in $\Delta(JM(A))$ and $\Delta(M(D))$ by Theorem \ref{ubsj}.   Thus
$q A q \subset A^{\perp \perp} \cap D = A$.   Hence $D = qAq$. 
Similarly, $q^\perp  A q^\perp \subset A$, and  $\{ q  a q^\perp +  q^\perp a q : a \in A \} \subset A$.
Thus $A = W \oplus D$ as desired.

We have $\theta^{**}(q) = q$.  
Let $\pi = \theta_{|D}$.   We have as before that  $\theta(D) \subset D = \theta^2(D) \subset \theta(D),$
so that $\pi(D) =  \theta(D) = D$.    Since  $\theta^{**}(q) = q$, 
that is $(\pi)^{**}$ is unital, $\pi$ is a Jordan homomorphism as we said at the start of Section \ref{bstsec}.  Since
$P(A)$  consists of the fixed points of $\pi$ in $D$, $P(A)$ is  a Jordan subalgebra of $A$.   
Moreover if $a \in A, b \in P(A)$ we have $$P(a \circ b) =P(q(a \circ b)q) =   P(qaq \circ b) = 
 \frac{1}{2}(qaq \circ b + \pi(qaq \circ b)) = P(a) \circ b,$$
since  $\pi$ is a Jordan homomorphism, and $b$ is fixed by $\pi$, and $P(a) \circ b = P(qaq)  \circ b$.
Since $P^{**}(1) = q$, by the uniqueness in Theorem \ref{hsaex}, $P$ must be the zeroing extension to $A$ of the map  $\frac{1}{2} (I + \pi)$ on $D.$

For the converse, note that  if $a \in {\mathfrak r}_A$ then $qaq + (q aq)^* = q(a+a^*)q \geq 0$.  Hence 
$P(qaq) + P(q aq)^* \geq 0$ since $P =  \frac{1}{2} (I + \pi)$ on $D$ and both $I$ and $\pi$ are real positive on $D$
by Corollary \ref{trivj}.  Since $P$ annihilates
$W$ it is now clear that $P$ is real positive on $A = D + W$.  We leave the rest as an exercise.  
\end{proof}

{\bf Remarks.}  1) \ One may write $P$ in the previous result more explicitly.   Namely, $$P = \frac{1}{2} (I + \theta) 
=  \frac{1}{2} (I + \pi(\cdot) (2q-1)).$$
Here $\theta, q, u = 2q-1$ are as above, and $\pi : A \to Au$ is the isometric surjective Jordan homomorphism  coming from
the Banach-Stone theorem \ref{ubsj}, namely $\pi = \theta(\cdot) u$.   If $A$ is an (associative) operator algebra
then $A u = A$, also by that theorem.    Conversely it is easy to show,  as in 
\cite[Theorem 3.7]{BNp}, that a map of the form at the end of the last centered equation, is a symmetric 
projection on $A$ under reasonable conditions.     As we said elsewhere, there seems to be an error  in the
last line of the statement of \cite[Theorem 3.7]{BNp}: the conditions there do not imply that $P$ is real 
positive.

\smallskip

2) \  Suppose  that $P$ is a symmetric projection on a unital Jordan operator algebra $A$ and that 
$q = P(1)^*$.   Then  $P(a) q P(a) = P(P(a)^2)  \in P(A)$, 
and so  $P(A)$ is a Jordan subalgebra of $A$ 
 if and only if  $P(a) (1-q) P(a) = 0$ for all $a \in A$.    
 If further $q$ is hermitian (which it is e.g.\ if $P$ real positive) then $q$ is a projection.   

To prove these assertions, let $u = 2q^*-1$, and let $\theta = 2P-I$.  By \cite[Corollary 2.8]{AS}, $\theta(a^2) = \theta(a 1 a) = \theta(a) u^*  \theta(a)$ for $a \in A$.    That is,
 $$2P(a^2) - a^2= (2P(a) - a) u^*  
(2P(a) - a).$$  This formula contains a lot of information.  In particular, replacing $a$ by $P(a)$, we have $2 P(P(a)^2) - P(a)^2 = 2 P(a) q P(a) 
 - P(a)^2$.
Hence $P(a) q P(a) = P(P(a)^2)  \in P(A)$.    Thus $P(a)^2 \in P(A)$  if and only if 
$P(a)^2 =  P(a) q P(a)$.   It follows that $P(A)$ is a Jordan subalgebra of $A$ 
 if and only if  $P(a) (1-q) P(a) = 0$ for all $a \in A$. 
By \cite[Corollary 2.8]{AS}, $u$ is a unitary in the $JC^*$-algebra $A \cap A^*$, and in the $C^*$-algebra
generated by $A$.   If $q$ is hermitian then $u$ is a selfadjoint unitary, which forces $q$ to be a
projection.  

 The condition in the last 
paragraph that $P(a)(1- q) P(a) = 0$ for all $a \in A$  happens for example if $q$ is a projection and  $P(A) \subset qAq$,
which is the case in much of our paper.

\begin{corollary} \label{triv}  Let $A$ be an approximately unital Jordan operator algebra,
and let $P : A \to A$ be a symmetric projection which is approximately unital (that is, takes a Jordan cai to a Jordan cai,
or more generally for which $P^{**}$ is unital). 
Then $P$ is real positive, 
the range of $P$  is a Jordan subalgebra of $A$, and $P$ is a Jordan conditional expectation.   Also, 
$P = \frac{1}{2} (1 + \theta)$ for a period 2 surjective isometric unital Jordan
homomorphism $\theta : A \to A$.   \end{corollary}

\begin{proof}  
Here $P^{**}(1) = 1$ and the results follow from Theorem \ref{trivch2c}.
 \end{proof}

In the classification of symmetric projections in the selfadjoint  case (Theorem \ref{goq}), $q = P^{**}(1)$ is central.
In our setting of  Jordan operator algebras, if one insists that $q  = P^{**}(1)$ be central (that is, if  $q \circ a = qaq \in A$ for all $a \in A$)
then one obtains a 
characterization closely resembling  Theorem \ref{goq}:

\begin{theorem} \label{goqc}  Suppose that $P : A \to A$ is a   symmetric real positive 
projection on an approximately unital Jordan operator algebra $A$, and let
$q = P^{**}(1)$.  Then   $q$ is  is central if and only if $D = \{ a \in A : a = q a q \}$ is a Jordan ideal in $A$.
If these hold then $q \in M(A)$, the multiplier algebra defined at the start of {\rm \cite[Section 2.8]{BWj}}, 
 $P = 0$ on $q^\perp A q^\perp$, and there exists a period $2$ Jordan $*$-automorphism $\theta$ of $D = qAq$ 
so that $P = \frac{1}{2}(I + \theta)$ on $D$.    
Conversely, any  $P$ satisfying the conditions in the last sentence 
is  a symmetric real positive projection.   Finally, $P(A)$ is a Jordan subalgebra of $A$, and 
$P$ is a Jordan conditional expectation. 
\end{theorem}

\begin{proof}   We said earlier that $q$ was open.
 The  first `if and only if' then follows from \cite[Corollary 3.26]{BWj}, and the bijective correspondence between
HSA's and open projections from \cite{BWj,BNj}.     We know from 
Theorem \ref{ubsj}  that $u = (2P^{**} - I)(1) = 2q - 1 \in M(D)$ if
$D$ is a   $C^*$-algebra generated by $A$.   If $q$ is central then we have $q a = qaq \in A^{\perp \perp} \cap D = A$, since $q = (1+u)/2 \in M(D)$. 
So $q \in M(A)$  (and in $JM(A)$).    The assertions about $q^\perp A q^\perp$ and $\theta$, and the last line of the statement
follow from Theorem \ref{trivch2c}.   Finally the converse follows just as in Theorem \ref{goq}.  
\end{proof}

Note that there exist symmetric projections $P : A \to A$ on a unital operator 
algebra with $P(1)$ a projection $q$ and $P(A)$ a Jordan subalgebra, but $P$ is not real positive
and $P(A)$ is not contained in $qAq$ and  $P$ does not  kill each 
$q^\perp a + a q^\perp$ for $a \in A$.  
  However, if $P$ is a contractive projection into $qAq$ then 
$P$ extends to a unital positive map on $A + A^*$ 
by e.g.\ \cite[Lemma 1.3.6]{BLM}, and so $P$ is real positive.

We also mention a nonunital example: the projection $P$ on the upper triangular $2 \times 2$ matrices
taking the matrix with rows
$[a \; b]$ and $[0 \; c]$ to the matrix with rows
$[(a-c)/2 \;  \; 0]$ and $[0 \;  \; (c-a)/2]$ .  This is symmetric and extends to
a completely contractive projection on the containing $C^*$-algebra, but its range
is not a Jordan subalgebra.  In this example $P(1) = 0$.

\bigskip

  {\bf Remark.}   Related to the last theorem we mention that there is  a typographical error  in the statement of Lemma 3.6  in \cite{BNp}: of course $P'$ is completely contractive and bicontractive
there.

\section{Bicontractive projections} 

Suppose that $A$ is an approximately unital  Jordan operator algebra,
and $P : A \to A$ is a bicontractive real positive projection.
Then $P^{**}(1) = q$ is an  projection in $A^{**}$ by the proof of \cite[Lemma 3.6]{BNp}, indeed $q$ is an open projection
by Lemma {\rm \ref{pretr2}},
and indeed all of the conclusions of  Lemma {\rm \ref{pretr2}} and Theorem   {\rm \ref{tr2}} hold. 
However as pointed out in   \cite{BNp} before Example 6.1 there,  the Jordan variants of the `bicontractive projection problem'
are not going to be
any better; the range of the 5 by 5 counterexample listed before Lemma 4.7  in  \cite{BNp} is not a Jordan subalgebra.  
 As mentioned earlier,  we follow the approach taken in \cite{BNp} that the correct formulation of the 
bicontractive projection problem in our category is: when is the range of $P$ a Jordan subalgebra of $A$? 
In this section we will give a natural hypothesis under which the bicontractive projection problem
for  operator algebras and  Jordan operator algebras does have a nice solution.

We recall the `three step reduction to the unital case' from above Corollary 4.3 in \cite{BNp}:  If $P$ and $A$ are as above, 
then by considering $P^{**} : A^{**} \to  A^{**}$ we may assume that $A$ is unital.
The second step is to use  Lemma {\rm \ref{pretr2}} and Theorem   {\rm \ref{tr2}} to reduce to the 
case that $P$ is unital (that is, we replace unital  $A$ by $qAq,$ 
where  $q  = P(1) = P(q)$).    The third step 
is to replace the domain $A$ of $P$ by joa$(P(A))$.    This does not change the range, the structure on which is our interest.   

Thus henceforth in this section we will be assuming that  $A$ is a unital  operator algebra or  Jordan operator algebra, and  $P:A \rightarrow A$ is a unital 
projection on $A$.

The following is a Jordan version of part of  \cite[Lemma 4.1]{BNp}.  

 \begin{lemma}
\label{41bn}  Let $A$ be a unital Jordan operator algebra, and suppose that  $P:A \rightarrow A$ is a unital 
projection on $A$ with $I-P$ contractive. 
  Let $C = {\rm Ker}(P)$ and $B = P(A)$.   Then  $\theta = 2P-I$ 
is the map $b + c \mapsto b-c$ for $b \in P(A), c \in {\rm Ker}(P)$, and  $P(A)$ is the set of fixed points of 
$\theta$.   We have $x^2 \in B$ for $x \in C$. If  $\theta$ is a Jordan homomorphism then $P(A)$ is a subalgebra. 
 Also 
\begin{enumerate}
\item $C$ is a Jordan subalgebra of $A$ if and only if $C$ has the zero Jordan product.
\item   $P(a \circ P(b)) = P(P(a) \circ P(b))$ for all $a, b \in A$ if and only if  $c \circ b \in C$ for $b \in B, c \in C$.
\item  $C$ is a Jordan ideal in $A$  if and only if  the conditions in {\rm (1)} and  {\rm (2)}  hold. 
 (As we said in Lemma {\rm \ref{keride}}, these are also equivalent to 
$P$ being a Jordan homomorphism with respect to the $P$-product on $P(A)$.) 
\item  If $c \circ b \in C$ for $b \in B, c \in C$ (see {\rm (2)}), then $\theta = 2P-I$ is a Jordan homomorphism on $A$ if and only if $P(A)$ is a Jordan 
subalgebra of $A$.  
\end{enumerate}
\end{lemma} 

\begin{proof}  The first assertions are an exercise.  For example, if $\theta = 2P-I$ is a Jordan homomorphism on $A$  then its fixed points, namely $P(A)$, form a Jordan subalgebra.   If $I-P$ is contractive then by part of the proof of
Lemma \ref{Araz} but with $P$ replaced by $I-P$, we have
$(I-P)(x^2) = (I-P)(x ((I-P)(1))^* x) = 0$ for $x \in C$.  Thus
$x^2 = P(x^2) \in B$.  

 For (1), if $C$ is a Jordan subalgebra  and $x \in C$ then $x^2 \in C \cap B = (0)$.
Item (2) is obvious (and does not use the statement before (1) or that $I-P$ is contractive).

%One direction of the first equivalence in
Item   (3) follows from (1) and (2) and the fact that $C \circ A \subset C$ if and only if  $C \circ B \subset C$  and  $C \circ C \subset C$.
For (4),  if $P(A)$ is a subalgebra then  
$$\theta((c +b)^2) = \theta(b^2) + \theta(c^2) - 2 \theta(c \circ b) = \theta(b^2) + c^2 - cb - bc =  c^2 + b^2 - cb -bc,$$   and
 $\theta(c +b)^2 = c^2 + b^2 - cb -bc$.     So $\theta$ is a Jordan homomorphism.  \end{proof}

\begin{lemma} \label{min2}  Let $A$ be a unital operator space.    If $a \in \frac{1}{2} \mathcal{F}_{A}$ and $a \neq 0$ then there exists an $\epsilon > 0$ such that $\| a - t 1 \| < \| a \|$ for all $0 < t < \epsilon$.  \end{lemma}

\begin{proof}   We have $a^* a \leq {\rm Re}(a)$, and so if $0 < t \leq \frac{1}{2}$ then we have 
$$(a - t 1)^* (a - t 1) \leq a^* a -2 t  {\rm Re}(a) + t^2 1 \leq  ((1-2 t)  a^* a + t^2 1 \leq ((1-2t) \| a \|^2  + t^2) 1.$$
Thus $\| a - t1 \|^2 \leq (1-2t) \| a \|^2  + t^2$.    The latter quantity is dominated by $\| a \|^2$ if and only if  $t < 2 \| a \|^2$. 
  Hence the result follows if $\epsilon = \min \{  \frac{1}{2} , 2 \| a \|^2 \}$.   \end{proof}

We thank the referee for a correction in the last proof.

We recall from the introduction that  Re $(A) = \{ {\rm Re}(a) : a \in A \}$, and that this is well-defined independently of the 
Hilbert space representation of $A$ if $A$ is a unital operator space or approximately unital 
Jordan operator algebra.    

\begin{theorem} \label{bic2}
Let $A$ be a unital  operator algebra (resp.\  Jordan operator algebra), and let $D = {\rm Ker}( P) \cap {\rm oa}(P(A))$ (resp.\ ${\rm Ker}( P) \cap {\rm joa}(P(A))$). 
 If $P:A \rightarrow A$ is a contractive  unital projection on $A$,  with $I-P$ contractive  on $A$ or with $I-P$ contractive  on ${\rm Re}(A)$, and if 
$D$  is
densely spanned by the real positive elements which it contains, then $P(A)$ is a subalgebra (resp.\ Jordan subalgebra) of $A$.  
\end{theorem}

\begin{proof}  By replacing $A$ with ${\rm oa}(P(A))$  (resp.\  ${\rm joa}(P(A))$) we may assume that these sets are the same 
 and $D = {\rm Ker}( P)$.   We use ideas found in Lemma 5.1 of  \cite{BNp}.  
If $a \in D \cap \frac{1}{2} \mathcal{F}_{A}$  and  $\| a - t 1 \| < \| a \|$ for some $t > 0$ (see Lemma \ref{min2}) then we obtain the contradiction 
$$\| a - t 1 \|  < \| a \| = \| (I-P) (a - t 1) \|.$$  Hence $D \cap \mathcal{F}_{A}=\{0\}$.   
By our hypothesis, which implies that $D$ is spanned by $D \cap \mathcal{F}_{A}$ (see e.g.\ Lemma \ref{genf} (1)), we have $D = \{0\}$.  So $P$ is the identity map, and our conclusion is tautological. 

Suppose that $I-P$  is contractive  on ${\rm Re}(A),$ that $x \in D \cap \mathcal{F}_{A}$, and write $x = a+ib$ with $a \in {\rm Re}(A)_+$ and $b = b^*$.
Continuing to write $P$ for its positive extension to $A+A^*$ (see \cite[Corollary 4.9]{BWj}), we have  $0 = P(x) = P(a) + i P(b)$ with $P(a), P(b)$ selfadjoint, 
so $P(a) = 0$.  By an argument similar to the last paragraph (also found in the proof of Lemma 5.1  in \cite{BNp}) we see that 
$a = 0$.  Hence $x = 0$ since $x \in {\mathfrak F}_A$.
Hence $D \cap \mathcal{F}_{A}=\{0\}$, and we finish as before.    
 \end{proof}

{\bf Remark.}  As in the first Remark after Theorem \ref{tr}, one may `weaken' the condition  in Theorem  \ref{bic2} about being `densely spanned by the real positive elements which it contains',
to being `contained in the closed algebra  (or even HSA)  generated by the real positive elements it contains'.

\medskip

{\em Acknowledgments.}  We thank the referee of \cite{BNp} for suggesting that we extend some of the results of that paper to contractive projections on 
Jordan operator algebras.   
 Part of the present paper is referred to as reference [13] in both \cite{BWj} and \cite{BNj}, and it served as motivation for the latter work.
  We thank Jonathan Arazy for several communications relative to a question that the first author asked him
(see also the remark after Lemma \ref{Araz}).  We
also thank Zhenhua Wang and Worawit Tepsan  
for questions while sitting through presentations of some of this material.     Finally we thank the referee of the present paper for an absolutely
 superb and remarkably detailed report: finding innumerable typos and a few actual mistakes, and supplying some corrections, alternative arguments, improvements,  interesting remarks, and important references from the literature, etc.   We have  indicated the most significant 
 referee corrections  in the text.  As an example, we mention the Remark after Lemma \ref{kerlem}.


\begin{thebibliography}{99} 

 \bibitem{Ake2}   C. A. Akemann, {\em Left ideal structure of $C^*$-algebras,}  J.\ Funct.\ Anal.\
\textbf{6}  (1970), 305-317.

 \bibitem{AP} C. A. Akemann and G. K. Pedersen, {\em Complications of semicontinuity
in $C^*$-algebra theory,}  Duke Math. J. {\bf 40} (1973), 785--795.

\bibitem{AS}   J. Arazy and B. Solel, {\em 
Isometries of nonselfadjoint operator algebras,} 
J. Funct. Anal. {\bf 90} (1990),  284--305. 

\bibitem{Arv}   W. B. Arveson, {\em Subalgebras of $C^{*}$-algebras,}  Acta Math.\  {\bf 123} (1969), 141--224.  

\bibitem{BBS}  C. A. Bearden, D. P. Blecher and S. Sharma, {\em On positivity and roots in operator algebras,}  Integral Equations Operator Theory {\bf  79} (2014),  555--566.


\bibitem{Bla}  B. Blackadar, {\em Operator algebras. Theory of $C^*$-algebras and von Neumann algebras,} Encyclopaedia of Mathematical Sciences, 122. Operator Algebras and Non-commutative Geometry, III. Springer-Verlag, Berlin, 2006. 


\bibitem{Bare}  D. P. Blecher,  {\em Are operator algebras Banach algebras?}  In: Banach algebras and their applications, pp.\ 
53--58, Contemp. Math., 363, Amer. Math. Soc., Providence, RI, 2004.





\bibitem{Bsan} D. P. Blecher,  {\em Generalization of C*-algebra methods via real positivity for operator  and Banach algebras,}  
 pages 35--66
in "Operator algebras and their applications: A tribute to Richard V. Kadison", (ed.\ by R.S.\ Doran and E.\ Park),
vol. 671, Contemporary Mathematics, American Mathematical Society, Providence, R.I.\  
2016.

\bibitem{Brp} D. P. Blecher,  {\em $C^*$-algebra valued projections on operator algebras} (tentative title), in preparation (2019). 

\bibitem{BHN}  D. P. Blecher, D. M. Hay, and
M. Neal, {\em Hereditary subalgebras of operator algebras,} J.\
Operator Theory {\bf 59} (2008), 333-357.

\bibitem{BLM}  D. P. Blecher and 
C.  Le Merdy, {\em Operator algebras and their modules---an
operator space approach,} Oxford Univ.\  Press, Oxford (2004).



\bibitem{BM05}  D. P. Blecher
and B. Magajna, {\em Duality and operator algebras II: Operator algebras as Banach algebras}, J. Funct.\ Analysis 
{\bf 226} (2005), 485--493.

\bibitem{BNII} D. P. Blecher
and M. Neal, {\em Open projections in operator algebras II: Compact projections,} Studia Math.\ {\bf 209} (2012), 203--224.

 
\bibitem{BNp}   D. P. Blecher
and M. Neal, {\em Completely contractive projections on operator algebras,}  Pacific J. Math.   {\bf 283-2} (2016), 289--324.


\bibitem{BNj}   D. P. Blecher
and M. Neal, {\em  Noncommutative topology and Jordan operator algebras,}  Mathematische Nachrichten {\bf 
292} (2019), 481--510


\bibitem{BRI}  D. P. Blecher and C. J. Read, {\em  Operator algebras with contractive approximate identities,}
J. Functional Analysis {\bf 261} (2011), 188--217.

\bibitem{BRII}  D. P. Blecher and C. J. Read, {\em  Operator algebras with contractive approximate identities II,} J. Functional Analysis {\bf 264} (2013), 1049--1067.

\bibitem{BRord}  D. P. Blecher and C. J. Read, {\em   Order theory  and interpolation in operator algebras,}   Studia Math. {\bf 225} (2014), 61--95.

\bibitem{BRS}  D. P. Blecher, Z-J. Ruan, and A. M. Sinclair, {\em   A characterization of operator algebras}, J.\ Functional Analysis {\bf 89} (1990), 288--301.

\bibitem{BWj}  D. P. Blecher and Z. Wang, {\em Jordan operator algebras: basic theory,} Mathematische Nachrichten, {\bf 291} (2018), 1629--1654.

\bibitem{BWj2}  D. P. Blecher and Z. Wang, {\em Jordan operator algebras revisited,}
 Mathematische Nachrichten {\bf 292} (2019), 2129--2136.

\bibitem{BKU} R. B. Braun, W. Kaup, and H. Upmeier, {\em A holomorphic characterization of Jordan $C^*$-algebras,}  Math. Z.\ {\bf 161} (1978), 277--290.

 \bibitem{Brown} L. G. Brown, {\em Semicontinuity and multipliers of $C^*$-algebras,} Canad. J. Math. {\bf 40} (1988),  865--988. 

 \bibitem{Brown2} L. G. Brown, {\em Close hereditary $C^*$-subalgebras and the structure of quasi-multipliers,} Proc. Roy. Soc. Edinburgh Sect. A
{\bf 147} (2017),  263--292.



\bibitem{BFT}  L. J. Bunce, B. Feely, and R. M. Timoney, {\em Operator space structure of $JC^*$--triples and TROs, I,} Math. Z. {\bf 270} (2012), 
961--982. 

\bibitem{BT}  L. J. Bunce and R. M. Timoney, {\em Universally reversible $JC^*$-triples and operator spaces,} J. Lond. Math. Soc. {\bf  88} (2013), 271--293.  


\bibitem{Rod}  M. Cabrera Garc\'ia and A.  Rodr\'iguez Palacios, {\em 
	Non-associative normed algebras. Vol. 1,}
The Vidav-Palmer and Gelfand-Naimark theorems. Encyclopedia of Mathematics and its Applications, 154. Cambridge University Press, Cambridge, 2014.

\bibitem{Rod2}  M. Cabrera Garc\'ia and A.  Rodr\'iguez Palacios, {\em 
	Non-associative normed algebras. Vol. 2, Representation Theory and the Zel'manov Approach}. Encyclopedia of Mathematics and its Applications, 167. Cambridge University Press, Cambridge, 2018.


\bibitem{CE}  M.-D. Choi and E. G. Effros, {\em Injectivity and
operator spaces,} J.\ Funct.\ Anal.\  {\bf 24}
(1977), 156--209.



\bibitem{ES}
E. G. Effros and E. St{\o}rmer, \emph{Positive projections and Jordan structure in operator algebras,} Math. Scand. 45 (1979),  127--138. 



  
\bibitem{FRAD}  Y. Friedman and B.  Russo, {\em
Conditional expectation without order,}  Pacific J. Math. 
{\bf  115} (1984), 351--360.

\bibitem{CPP}  Y. Friedman and B.  Russo,  {\em  Solution of
the contractive projection problem,} J.\ Funct.\ Anal.\  {\bf 60}
(1985), 56--79.

\bibitem{FRceJ}  Y. Friedman and B.  Russo,  {\em  Conditional expectation and bicontractive projections on Jordan $C^*$-algebras and their generalizations,} Math.\ Z.\ {\bf  194} (1987), 227--236. 


\bibitem{Ham}  M. Hamana,  {\em Triple envelopes and Silov boundaries of operator spaces,} Math. J. Toyama Univ. {\bf 
22} (1999),  77--93.

\bibitem{HS}  H. Hanche-Olsen and E. St{\o}rmer, {\em  
Jordan operator algebras,} 
Monographs and Studies in Mathematics, 21. Pitman (Advanced Publishing Program), Boston, MA, 1984.

\bibitem{Har1}   L. A. Harris, {\em  Bounded symmetric homogeneous
domains in infinite dimensional spaces,}   Lecture Notes in
Math., {\bf 364}, Springer-Verlag, Berlin - New York (1973).
 
\bibitem{Har2}  L. A. Harris, {\em  A generalization of $C^*$-algebras,} Proc. Lond. Math. Soc. {\bf 42}  (1981),  331--361.

\bibitem{Hay}  D. M. Hay, {\em Closed projections and peak interpolation for operator algebras,}
  Integral Equations Operator Theory  {\bf 57}  (2007),  491--512.  


\bibitem{IR} J. M.  Isidro and A.  Rodr\'iguez-Palacios, {\em Isometries of JB-algebras,} Manuscripta Math. {\bf 86} (1995), 337--348. 

 \bibitem{Kad}   R. V. Kadison, {\em  Isometries of
operator algebras,} Ann.\ of Math.\ {\bf 54} (1951), 325--338.

\bibitem{KP} M. Kaneda and V. I. Paulsen, {\em 
Quasi-multipliers of operator spaces,} J.\ Funct.\ Anal.\ {\bf 217} (2004), 347--365. 


\bibitem{Kaup} W. Kaup,  {\em Contractive projections on Jordan $C^*$-algebras and generalizations,} Math. Scand.{\bf  54},  (1984), 95--100.


\bibitem{LL} A. T. Lau and  R. J. Loy, {\em Contractive projections on Banach algebras,} J. Funct. Anal. {\bf 254} (2008), 2513--2533. 


\bibitem{M}  J. D. Maitland Wright, {\em Jordan $C^*$-algebras,}   Michigan Math. J. {\bf 24} (1977), 291--302.  


\bibitem{Mc} K. McCrimmon, {\em A taste of Jordan algebras,}  
Universitext. Springer-Verlag, New York, 2004.


\bibitem{N} M.  Neal, {\em Inner ideals and facial structure of the quasi-state space of a JB-algebra,}
 J. Funct. Anal.\ {\bf  173} (2000), 284--307.

\bibitem{N2} M.  Neal, E. Ricard and B. Russo, {\em Classification of contractively complemented Hilbertian operator spaces,} J.\ Funct.\ Anal.\  237 (2006), 589--616.


\bibitem{NR03}  M. Neal and B. Russo, {\em Contractive projections and operator spaces}, Trans. Amer. Math. Soc. {\bf 355} (2003), 2223--2262. 


%\bibitem{NeRu} M. Neal and B. Russo, {\em Operator space characterizations of $C^*$-algebras and ternary rings}, Pacific J. Math. {\bf 209} (2003), 339--364.

\bibitem{NRh} M. Neal and B. Russo, {\em A holomorphic characterization of operator algebras}, Math. Scand. {\bf 115} (2014), 229-268.


\bibitem{Pau}   V. I. Paulsen, {\em Completely bounded maps and operator
algebras,} Cambridge Studies in Advanced Math., 78, Cambridge
University Press, Cambridge, 2002.

\bibitem{P} G. K. Pedersen, {\em $C^*$-algebras and their automorphism
groups,} Academic Press, London (1979).



\bibitem{RPa} A. Rodr\'iguez-Palacios, {\em Jordan structures in analysis,} In: Jordan Algebras (Oberwolfach, 1992), pp. 97--186 de Gruyter, Berlin (1994).


\bibitem{SOJ}  B. Russo, {\em Structure of JB*-triples,} In: Jordan
Algebras,
 Proc. of Oberwolfach Conference 1992, Ed. W. Kaup, et al., de Gruyter
(1994), 209--280.



\bibitem{ST82}  E. St{\o}rmer, {\em Positive projections with contractive complements on $C^*$-algebras,} J. London Math. Soc. {\bf 26} (1982),  132--142. 





\bibitem{ST}  E. St{\o}rmer, {\em Positive linear maps of operator algebras,} Springer Monographs in Mathematics, Springer-Verlag (2013). 




\bibitem{Up} H. Upmeier, {\em 
Symmetric Banach manifolds and Jordan $C^*$-algebras,} North-Holland Mathematics Studies, 104. Notas de Matematica [Mathematical Notes], 96. North-Holland Publishing Co., Amsterdam, 1985. 


\bibitem{Up87} H. Upmeier, {\em Jordan algebras in analysis, operator theory, and quantum mechanics}, CBMS Regional Conference Series in Mathematics, No. 67 (1987).


\bibitem{ZWthes}  Z. Wang, {\em Theory of Jordan operator algebras and operator $*$-algebras,} PhD thesis, University of Houston, 2019.

 \bibitem{Y}  M. A. Youngson, {\em Completely contractive projections on $C^*$-algebras}, Quart. J. Math. Oxford Ser. {\bf 34} (1983), 507--511.

\end{thebibliography}
\end{document}